\documentclass[10pt]{amsart}

\usepackage[english]{babel}

\usepackage{amsmath}
\usepackage{graphicx}
\usepackage[colorlinks=true, allcolors=blue]{hyperref}
\usepackage{amsthm}
\usepackage{amssymb}
\usepackage{amsfonts,amscd}
\usepackage[mathscr]{euscript}
\usepackage[final]{showkeys}
\usepackage{fullpage}
\usepackage{subfig}
\usepackage{enumitem}
\usepackage{arydshln}
\usepackage{xspace}
\usepackage{adjustbox}
\setlist[enumerate,1]{label={\rm (\roman*)},leftmargin=2.5em}

\usepackage{tikz}
\usetikzlibrary{cd,positioning,automata,arrows,shapes,matrix, arrows, decorations.pathmorphing}
\tikzset{main node/.style={circle,draw,minimum size=0.3em,inner
sep=0.5pt}}
\tikzset{state node/.style={circle,draw,minimum size=2em,fill=blue!20,inner
sep=0pt}}
\tikzset{small node/.style={circle,draw,minimum size=0.5em,inner sep=2pt,font=\sffamily\bfseries}}

\theoremstyle{plain}
\newtheorem{theorem}{Theorem}[section]
\newtheorem{lemma}[theorem]{Lemma}
\newtheorem{prop}[theorem]{Proposition}

\newtheorem{cor}[theorem]{Corollary}
\theoremstyle{definition}
\newtheorem{defn}[theorem]{Definition}
\newtheorem{exa}[theorem]{Example}

\theoremstyle{remark}
\newtheorem{rmk}[theorem]{Remark}
\numberwithin{equation}{section}
\newcommand{\al}{\alpha}
\newcommand{\be}{\beta}
\newcommand{\ep}{\varepsilon}
\newcommand{\Q}{{\mathbb Q}}
\newcommand{\R}{{\mathbb R}}
\newcommand{\Z}{{\mathbb Z}}

\newcommand{\FF}{{\mathbb F}}
\newcommand{\kp}{\kappa}
\newcommand{\ra}{\rightarrow}

\newcommand{\QQ}{\mathcal Q}

\newcommand{\fg}{{\mathfrak g}}
\newcommand{\fk}{{\mathfrak k}}
\newcommand{\fp}{{\mathfrak p}}

\newcommand{\simproots}{\Pi}
\newcommand{\posroots}{{\Phi_+}}
\newcommand{\allroots}{\Phi}

\newcommand{\Span}{\text{\rm Span}}
\newcommand{\rank}{\text{\rm rank}}

\newcommand{\rootht}{\text{\rm ht}}
\newcommand{\lht}{\rho}

\newcommand{\xht}{\text{\rm ht}}

\newcommand{\res}{\mathscr{D}}
\newcommand{\resi}[1]{\mathscr{D}^{(#1)}}

\newcommand{\ihtj}[2]{\Phi_{{#1},{#2}}}
\newcommand{\dcup}{\ \dot\cup \ }

\newcommand{\N}{\mathbb{N}}

\newcommand{\qpf}{\text{\rm QPf}}
\newcommand{\pf}{\text{\rm Pf}}

\newcommand{\qhf}{\text{\rm QHf}}
\newcommand{\hf}{\text{\rm Hf}}

\newcommand{\per}{\text{\rm per}}
\newcommand{\omu}[1]{\Omega_{#1}}
\newcommand{\upr}{{U'}}
\newcommand{\blm}{{\lambda}^-}

\newcommand{\se}{\subseteq}

\newcommand{\inverse}{^{-1}}

\newcommand{\abs}[1]{\lvert #1 \rvert}

\makeatletter
\newcommand{\exterior}[1]{\mathop{\mathpalette\exterior@{#1}}}
\newcommand{\exterior@}[2]{%
  \raisebox{\depth}{%
  \fontsize{\sf@size}{0}%
  \m@th
  $\ifx#1\displaystyle\textstyle\else#1\fi\bigwedge$}%
  ^{\mspace{-2mu}#2}%
  \kern-\scriptspace
}
\makeatother

\title{Generalized Rothe diagrams for orthogonal roots}

\author{R.M. Green}
\address{Department of Mathematics, University of Colorado Boulder, Campus Box
395, Boulder, Colorado, USA, 80309}
\email{rmg@colorado.edu}

\author{Tianyuan Xu}
\address{Department of Mathematics and Statistics, University of Richmond,
Richmond, Virginia, USA, 23173}
\email{tianyuan.xu@richmond.edu}


\begin{document}

\begin{abstract}
Let $U$ be a set of positive roots of type $ADE$, and let $\Omega_U$ be the set
of all maximum-cardinality orthogonal subsets of $U$. We associate a
generalized Rothe diagram to each element $R\in \Omega_U$ as a broad,
root-theoretic generalization of the traditional Rothe diagrams of permutations,
and we use the generalized Rothe diagrams to define a $q$-polynomial in $U$ that
we call the generalized quantum Hafnian of $U$. 

We study a large number of examples where these constructions recover a variety
of widely studied algebraic and combinatorial objects. One of our motivating
examples involves a certain set $U$ of $k^2$ roots in type $D_{2k}$, where the
elements of $\Omega_U$ can be identified with permutations in $S_k$, the
generalized Rothe diagrams are the traditional Rothe diagrams of permutations,
and the generalized quantum Hafnian is the $q$-permanent. In another example,
the generalized quantum Hafnian gives a non-recursive method to compute the 45
terms of a well-known invariant cubic polynomial of type $E_6$. More generally,
all our examples in types $A$ and $D$ are closely related to perfect matchings
and rook configurations, and our examples in type $E$ have applications to
labelled Fano planes, del Pezzo surfaces, and minuscule representations. Each of
our examples also gives rise to a matroid, and many of our examples have an 
associated equal-rank simply-laced symmetric pair.
\end{abstract}

\maketitle

\section{Introduction}\label{sec:intro}

Let $\Phi$ be a root system of type $ADE$ with Weyl group $W$. Let $U$ be a
subset of the positive roots in $\Phi$, and let
$\Omega_U$ be the set of orthogonal subsets of $U$ of maximum cardinality. For
each $R \in \Omega_U$, we say that a root $\al\in U$ \emph{dominates} $R$, or
that $\al$ is \emph{$R$-dominating}, if $s_\be(\al)$ is positive for all $\be
\in R$, where $s_\be$ denotes the reflection associated to $\be$. We define the
{\it generalized Rothe diagram}, $\res_U(R)$, to be the set of all
$R$-dominating roots in $U$, and we call $\lht(R) := |\res_U(R)|$ the {\it
level} of $R$. We define the {\it generalized quantum Hafnian} of $U$ to be the
element of $\Z[U][q]$ given by 
\[
  \qhf(U):=
  \sum_{R \in \Omega_U} q^{\lht(R)} \prod_{\be
\in R} x_{\be} ,\] where the elements $x_\be$ are commuting indeterminates indexed by
the set $U$.

Although it is not {\it a priori} clear why one might want to study these
definitions, which are new to the best of our knowledge, we will show that for
many choices of $U$, the elements of $\Omega_U$ and their generalized Rothe
diagrams give rise to concise constructions of known algebraic and combinatorial
objects, in some cases endowing them with additional structure. For instance,
our examples in types $A$ and $D$ are naturally related to perfect matchings and
rook configurations, and our examples in type $E$ are related to labellings of
Fano planes, curves on del Pezzo surfaces, and invariant polynomials of
minuscule representations. In all of our examples, the level function
$\lht$ also endows $\Omega_U$ with the structure of a quasiparabolic set for a
suitable parabolic subgroup $W_I$ of $W$ in the sense of Rains--Vazirani
\cite{rains13}, so that $\Omega_U$ forms a graded poset with a Bruhat-like
partial order with rank function $\lht$ and can be used to construct a
module for the Iwahori--Hecke algebra of $W$ (Remark \ref{rmk:deform}). 

We highlight the key properties of some of our examples in more detail
below, where $\Gamma$ denotes the Cartan type of the root system,
$S$ stands for the simple reflections of $W$, $n=\abs{S}$ is the rank of
$\Phi$, and $\Phi_+$ is the set of all positive roots of $\Phi$.

\begin{enumerate}
  \item For $\Gamma=D_{2k} (k\ge 2)$ and $U=\{\ep_i +
    \ep_{k+j} : 1 \leq i, j \leq k\}$, we can identify each root
    $\ep_i+\ep_{k+j}\in U$ with the position $(i,j)$ on a $k\times k$ grid. In 
    this case, each set $R\in \Omega_U$ corresponds to a configuration of $k$
    nonattacking rooks, or equivalently, to a permutation $\tau$ in the
    symmetric group $S_k$. The set $\res_U(R)$ can be naturally identified with
    a well-studied object known as the Rothe diagram associated to
    $\tau$ (see Figure \ref{fig:rothe_A} and Remark \ref{rmk:biset} (i)), the
    level $\rho(R)$ is the Coxeter length of $\tau$ (equivalently, the number of
    inversions of $\tau$), and the
    generalized quantum Hafnian $\qhf(U)$ recovers the $q$-permanent of a $k
    \times k$ matrix $A$. Specializing $q$ to $-q$ in $\qhf(U)$ gives the
    $q$-determinant of $A$. This is one of our motivating examples, and is
    studied in detail in Section \ref{sec:motexa}; see Theorem \ref{thm:perdet}.

\begin{figure}[h!]
\centering

\begin{tikzpicture}[scale=0.7]
    \draw (0,0) grid (4,4);

    \foreach \x in {1,...,4}
        \draw (\x-0.5, 4.3) node {$\x$};
    \foreach \y in {1,...,4}
        \draw (-0.3, 4.5-\y) node {$\y$};
        
    
    
    \node at (1-0.5, 4-1+0.5) {$\bigstar$};
    
    \node at (1-0.5, 4-2+0.5) {$\bigstar$};
    
    \node at (3-0.5, 4-2+0.5) {$\bigstar$};
    
    
    \filldraw[black] (1-0.5, 4-3+0.5) circle (4pt); 
    \filldraw[black] (2-0.5, 4-1+0.5) circle (4pt); 
    \filldraw[black] (3-0.5, 4-4+0.5) circle (4pt); 
    \filldraw[black] (4-0.5, 4-2+0.5) circle (4pt); 
\end{tikzpicture}
\caption{The Rothe diagram of $\underline{3142}\in S_4$}
\label{fig:og_rothe}
\end{figure}

  \item For $\Gamma=D_{2k} (k\ge 2)$ and $U=\{\ep_i + \ep_j : 1 \leq i < j \leq
    2k\}$, the set $\Omega_U$ corresponds to the set of perfect matchings of the
    set $\{1,2,\dots, 2k\}$ and also to the fixed-point-free involutions
    of the symmetric group $S_{2k}$. The generalized Rothe diagrams in this case
    recover the Rothe diagrams of fixed-point-free involutions as defined by
    Hamaker--Marberg--Pawlowski \cite[(3.3)]{hamaker18}, and the generalized quantum
    Hafnian is precisely the $q$-Hafnian defined by Jing--Zhang \cite{jing16}.
    Specializing $q$ to $-q$ gives the quantum Pfaffian defined by Strickland
    \cite{strickland96}. This is another motivating example that we treat in
    Section \ref{sec:motexa}; see Theorem \ref{thm:hafpfaf}.

  \item If $\Gamma$ is $D_{2k} (k\ge 2)$, $E_7$, or $E_8$, then for $U=\Phi_+$,
    the maximum cardinality of orthogonal subsets of $U$ is the rank $n$ of
    $\Phi$, and the elements of $\Omega_U$ recover the ``positive $n$-roots of
    $\Phi$" studied in \cite{gx5}. In particular, $\Omega_U$ spans a certain
    Macdonald representation of $W$ and contains as subsets certain noncrossing
    and nonnesting bases of the Macdonald representation; see \cite[Section
    5]{gx5}. In this sense, the present paper may be viewed as a generalization
    of \cite{gx5}.

  \item If $\al_p (1\le p\le n)$ is a simple root, $s_p$ is the corresponding
    simple reflection, and $U=U(\Gamma,p)$ is the set of all positive roots in
    which $\al_p$ appears with odd coefficient, then $\Omega_U$ forms a
    quasiparabolic set of $W_{I_p}$ for $I_p=S\setminus \{s_p\}$ in a large
    number of cases. The full list of these cases is listed in Section
    \ref{sec:p_triples}, and they include (but are not limited to) the
    following:

    \begin{enumerate}
      \item when $\Gamma=A_n$ and $p$ is any number in $\{1,2,\cdots, n\}$, or
        when $\Gamma=D_{n}$ and $p\in \{1,2,\dots, \lfloor n/2\rfloor\}$, both
        the elements of $\Omega_U$ and their generalized Rothe diagrams have
        interesting combinatorial interpretations in terms of rook
        configurations (Example \ref{exa:aiii}, Example \ref{exa:di});

      \item when $\Gamma\in \{E_7, E_8\}$ and $p=2$, the elements of $\Omega_U$ and
        their generalized Rothe diagrams can be interpreted in terms of
        labellings of Fano planes (Example \ref{exa:fano7}, Example
        \ref{exa:fano8}); 
      \item when $\Gamma\in \{E_7, E_8\}$ and $p=n$, i.e., when $U=U(E_7,7)$ or
        $U=U(E_8,8)$, the roots in $U$ can be identified with the weights of
        suitable minuscule representations, and the generalized Rothe diagrams
        for the elements of $\Omega_U$ can be used to compute invariant
        polynomials associated to the minuscule representations (Theorem
        \ref{thm:e77}, Remark \ref{rmk:bitangent}). 
\end{enumerate}
     
    \noindent The examples in case (a) include and generalize those from (i) and
    (ii), as we will explain in Section \ref{sec:p_triples}. The examples in (b)
    endow the Fano plane labellings in question with partial orders, as we have
    previously noted in \cite{gx5} and \cite{the240}. The examples in (c) are
    closely related to the quasiparabolic sets $\Omega_{\Phi_+}$ of types $E_7$
    and $E_8$ from (iii), and they can be obtained as quotients of other
    quasiparabolic sets via Reading's notion of poset congruence
    \cite{reading02}, by using a construction we introduce in Section
    \ref{sec:reading}. 

  \item For all the examples from (i)--(iv), every element $R\in \Omega_U$ gives
    rise to a binary matroid whose isomorphism type depends only on $U$ and not
    on $R$. Moreover, this matroid can be used to obtain a factorization of the
    level generating function $PS_U(q)=\sum_{R\in \Omega_U}q^{\rho(R)}$, known
    as the Poincar\'e polynomial of $\Omega_U$ in the theory of quasiparabolic
    sets, into quantum integers; see Remark \ref{rmk:matroid} and Example
    \ref{exa:matroid}. 

     \item The examples from (iv) correspond to Riemannian symmetric spaces whose
       associated symmetric pair consists of two simply-laced Lie algebras of
       the same rank, with $U$ being precisely the set of noncompact positive roots of
       the symmetric pair; see Remark \ref{rmk:symmspace} and Example
       \ref{exa:symmspace}.   \end{enumerate} 

In the above list, the examples that recover familiar objects such as
permutations, perfect matchings, and labelled Fano planes offer a new unifying
perspective on these objects in terms of orthogonal roots. This perspective may
help shed new light on these classical objects and suggest natural
generalizations. For example, in the paper \cite{gx7}, we will describe the
generalized Rothe diagrams for the set $\Omega_{\Phi_+}$ of type $E_8$ from (iii), 
we will explain how $\Omega_{\Phi^+}$ can be modelled using a
combination of perfect matchings and labelled Fano planes, and we will define a
set of inversions for each labelled Fano plane whose cardinality may be viewed
as a natural analogue of the Coxeter length of permutations. For another
example, the Rothe diagrams of permutations and of fixed-point-free involutions
mentioned in (i)--(ii) have applications to Schubert polynomials, and it would
be interesting to know whether the same is true for the other generalized Rothe
diagrams arising from the constructions in (iii)--(iv).

The sets $\Omega:= \Omega_{U(E_8,8)}$ and $\Omega':=\Omega_{U(E_7,7)}$ from
(iv)(b) are two main examples that we consider. As far as we know, our
combinatorial constructions of these objects have not appeared in the literature
before. The set $\Omega$ has the distinguishing property that it has the largest
cardinality among the quasiparabolic sets of the form $\Omega_{U(E_n,p)}$, and
the proof of the fact $\Omega$ is indeed a quasiparabolic set for
$W_{S\setminus \{s_8\}}$, given in Section \ref{sec:e88}, is the technical heart
of this paper. The set $\Omega$ is also helpful for understanding $\Omega'$,
which we study in full detail in Section \ref{sec:45}. We give explicit
descriptions of the generalized Rothe diagrams for the elements of $\Omega'$,
and we show that specializing $q$ to $-1$ in the generalized quantum Hafnian of
$U(E_7, 7)$ yields
a 45-term invariant cubic polynomial associated to a minuscule representation of
type $E_6$. This polynomial was previously calculated recursively using
\texttt{Mathematica} by Vavilov--Luzgarev--Pevzner \cite{vavilov07}, and the
fact that one can obtain it non-recursively via the quantum Hafnian may be
considered as a useful application of generalized Rothe diagrams. 

The rest of the paper is organized as follows. Section \ref{sec:oroots} recalls
the necessary background regarding the combinatorics of root systems. In Section
\ref{sec:rqs}, we give the key definitions of this paper, including Rains and
Vazirani's notion of quasiparabolic sets, and we explain details of the two
motivating examples described in (i) and (ii). We also recall  Reading's notion
of poset congruence and use it to define quotient quasiparabolic sets. Section
\ref{sec:further} discusses the constructions highlighted in (iii) and (iv),
with Section \ref{sec:e88} treating the example $\Omega_{U(E_8,8)}$ as a case
study in type $E$. We then use the results on $\Omega_{U(E_8,8)}$ to help study
the set $\Omega_{U(E_7,7)}$ in Section \ref{sec:45}. Finally, in Section
\ref{sec:conclusion}, we elaborate on the properties of $\Omega_U$ summarized in
(v)--(vi) and propose some possible directions for further research. 

\section{Orthogonal roots}\label{sec:oroots}

\subsection{Review of root systems}
\label{subsec:roots}
We recall the necessary background concerning root systems in this subsection.
While many definitions in this paper work more generally for all (finite) root
systems, we will focus on irreducible simply-laced root systems, i.e., on the
root systems of types $ADE$. We label the vertices in the Dynkin diagrams of
root systems as shown in Figure \ref{fig:ade} until Section \ref{sec:27}, where
we will switch to a different convention in type $E$ for reasons we will
explain.

\begin{figure}[!ht]
    \centering
    \subfloat[{$A_n$}]
    {
\begin{tikzpicture}

            \node[main node] (1) {};
            \node[main node] (2) [right=1cm of 1] {};
            \node[main node] (3) [right=1.5cm of 2] {};
            \node[main node] (4) [right=1cm of 3] {};

            \node (11) [below=0.1cm of 1] {\small{$1$}};
            \node (22) [below=0.1cm of 2] {\small{$2$}};
            \node (33) [below=0.1cm of 3] {\small{$n-1$}};
            \node (44) [below=0.15cm of 4] {\small{$n$}};

            \path[draw]
            (1)--(2)
            (3)--(4);

            \path[draw,dashed]
            (2)--(3);
\end{tikzpicture}
}
\quad\quad\quad
\subfloat[{$D_n$}]
{
\begin{tikzpicture}

            \node[main node] (1) {};
            \node[main node] (2) [right=1cm of 1] {};
            \node[main node] (3) [right=1.5cm of 2] {};
            \node[main node] (4) [right=1cm of 3] {};
            \node[main node] (5) [above right=0.7cm and 0.9cm of 4] {};
            \node[main node] (6) [below right=0.7cm and 0.9cm of 4] {};

            \node (11) [below=0.1cm of 1] {\small{$1$}};
            \node (22) [below=0.1cm of 2] {\small{$2$}};
            \node (33) [below=0.1cm of 3] {\small{$n-3$}};
            \node (44) [right=0.1cm of 4] {\small{$n-2$}};
            \node (55) [right=0.1cm of 5] {\small{$n-1$}};
            \node (66) [right=0.1cm of 6] {\small{$n$}};

            \path[draw]
            (1)--(2)
            (3)--(4)--(5)
            (4)--(6);

            \path[draw,dashed]
            (2)--(3);
\end{tikzpicture}
}\\
\subfloat[{$E_6$}]
{
\begin{tikzpicture}
    \node[main node] (1) {};
    \node[main node] (2) [right=1cm of 1] {};
            \node[main node] (3) [right=1cm of 2] {};
            \node[main node] (4) [right=1cm of 3] {};
            \node[main node] (5) [right=1cm of 4] {};
            \node[main node] (6) [above=1cm of 3] {};

            \path[draw]
            (1)--(2)--(3)--(4)--(5)
            (3)--(6);

            \node (11) [below=0.1cm of 1] {\small{$1$}};
            \node (22) [below=0.1cm of 2] {\small{$3$}};
            \node (33) [below=0.1cm of 3] {\small{$4$}};
            \node (44) [below=0.1cm of 4] {\small{$5$}};
            \node (55) [below=0.1cm of 5] {\small{$6$}};
            \node (66) [above=0.1cm of 6] {\small{$2$}};
\end{tikzpicture}
}
\quad\quad\quad
\subfloat[{$E_7$}]
{
\begin{tikzpicture}
    \node[main node] (1) {};
    \node[main node] (2) [right=1cm of 1] {};
            \node[main node] (3) [right=1cm of 2] {};
            \node[main node] (4) [right=1cm of 3] {};
            \node[main node] (5) [right=1cm of 4] {};
            \node[main node] (6) [right=1cm of 5] {};
            \node[main node] (7) [above=1cm of 3] {};

            \path[draw]
            (1)--(2)--(3)--(4)--(5)--(6)
            (3)--(7);

            \node (11) [below=0.1cm of 1] {\small{$1$}};
            \node (22) [below=0.1cm of 2] {\small{$3$}};
            \node (33) [below=0.1cm of 3] {\small{$4$}};
            \node (44) [below=0.1cm of 4] {\small{$5$}};
            \node (55) [below=0.1cm of 5] {\small{$6$}};
            \node (66) [below=0.1cm of 6] {\small{$7$}};
            \node (77) [above=0.1cm of 7] {\small{$2$}};
\end{tikzpicture}
}
\quad
\subfloat[{$E_8$}]
{
\begin{tikzpicture}
    \node[main node] (1) {};
            \node[main node] (3) [right=1cm of 1] {};
            \node[main node] (4) [right=1cm of 3] {};
            \node[main node] (5) [right=1cm of 4] {};
            \node[main node] (6) [right=1cm of 5] {};
            \node[main node] (7) [right=1cm of 6] {};
            \node[main node] (8) [right=1cm of 7] {};
            \node[main node] (2) [above=1cm of 4] {};

            \path[draw]
            (1)--(3)--(4)--(5)--(6)--(7)--(8)
            (2)--(4);

            \node (11) [below=0.1cm of 1] {\small{$1$}};
            \node (22) [above=0.1cm of 2] {\small{$2$}};
            \node (33) [below=0.1cm of 3] {\small{$3$}};
            \node (44) [below=0.1cm of 4] {\small{$4$}};
            \node (55) [below=0.1cm of 5] {\small{$5$}};
            \node (66) [below=0.1cm of 6] {\small{$6$}};
            \node (77) [below=0.1cm of 7] {\small{$7$}};
            \node (88) [below=0.1cm of 8] {\small{$8$}};

\end{tikzpicture}
}
\caption{Dynkin diagrams of irreducible simply-laced Weyl groups}
\label{fig:ade}
\end{figure}

Let $\Phi=\Phi(\Gamma)$ be a simply-laced root system with Dynkin diagram
$\Gamma$. The vertices of $\Gamma$ index the \emph{simple roots} $\simproots =
\{\alpha_i : i \in \Gamma\}$, and the \emph{rank} of $\Phi$ is
$\rank(\Phi):=\abs{\simproots}$. The {\it root lattice} $\Z\simproots$ is the
free $\Z$-module on $\simproots$, and it carries a symmetric $\Z$-bilinear form
$B$ given by
\begin{equation*}
     B(\alpha_i, \alpha_j) = 
    \begin{cases} 2 & \text{\ if \ } i = j;\\ -1 & \text{\ if \ $i$\ and\ $j$\
          are\ adjacent\ in\ } \Gamma;\\ 0 & \text{\ otherwise.}\\
    \end{cases} 
\end{equation*} 

If $\al_i \in \simproots$, then the {\it simple reflection} $s_i$ is the
self-inverse $\Z$-linear operator $\Z\simproots \rightarrow \Z\simproots$ given
by 
\begin{equation}
    \label{eq:s_i}s_i(\be) = \be - B(\al_i, \be)\al_i .
\end{equation}

The {\it Weyl group} $W = W(\Gamma)$ is the finite group generated by the set
$S=S(\Gamma)=\{s_i: i\in \Gamma\}$, and we have $\Phi=\{w(\al_i):\al_i\in \Pi,
w\in W\}$.  Each element of $\Phi$ is called a \emph{root}, and each root gives
rise to a \emph{reflection}, which is the self-inverse linear map $s_\alpha:
\Z\Pi\ra \Z\Pi$ given by the formula 
\begin{equation}\label{eq:refl} 
  s_\al(\be) = \be - B(\al, \be)\al,
\end{equation}
\noindent generalizing Equation \eqref{eq:s_i}.  We have $s_\al\in W$ for all
roots $\al\in \Phi$.  Moreover, the Weyl group acts transitively on the roots
and the reflections  $\{s_\al: \al\in \Phi\}$ form a single conjugacy class in
$W$.  The {\it reflection representation} of $W$ is the $\Q W$-module $V = \Q
\otimes_\Z \Z\simproots$, where each reflection acts by Equation
\eqref{eq:refl}. The simple roots form a basis of $V$, so that we have
$\dim(V)=\rank(\Phi)$.  

The pair $(W,S)$ is a Coxeter system, and the set of reflections $\{s_\al:\al\in
\Phi\}$ coincides with the set $T=\{wsw\inverse:s\in S, w\in W\}$.  For each
subset $I \subseteq S$, the {\it standard parabolic subgroup} $W_I \leq W$ is
the subgroup $\langle I \rangle$ generated by $I$.
    
Every root is of the form $\al = \sum_{i \in \Gamma} c_i \al_i$ for some
integers $c_i$ that are either all nonnegative or all nonpositive.  We say $\al$
is \emph{positive} or \emph{negative} in these cases, respectively. The sum of
the coefficients $\sum_{i\in \Gamma} c_i$ is called the \emph{height} of $\al$.
Given any $\al\in \Phi$, the only multiples of $ \al$ in $\Phi$ are $\al$ and
$-\al$.

There is a partial order on the root lattice $\Z\Pi$ defined by the condition
that $\al \leq \be$ if and only if $\be - \al$ is a nonnegative linear
combination of simple roots, so that a root $\al$ is positive if and only if
$\al>0$ in this order.  With respect to (the restriction of) this partial order,
$\allroots$ has a unique maximal element, $\theta$, called the {\it highest
root}.  We denote the set of positive roots $\{\al\in \Phi:\al>0\}$ by
$\posroots$.

A {\it subsystem} of a root system $\allroots$ is a subset $\Psi \subseteq
\allroots$ that is itself a root system. Equivalently, a subsystem of
$\allroots$ is a nonempty subset  that is closed under the operations $\{s_\al :
\al \in \Psi\}$. The positive roots of $\Phi$ naturally induce a \emph{positive
system}, $\Psi_+$, and a \emph{simple system}, $\Pi_\Psi$, in $\Psi$ (in the
sense of \cite[Section 1.3]{humphreys90}; see also \cite[Section 2.2]{gx5}). We
have $\Psi_+=\Psi\cap \Phi_+$, and $\simproots_\Psi$  can be characterized as
the set of elements of $\Psi_+$ that cannot be expressed as nonnegative linear
combinations of other elements of $\Psi_+$. The elements of $\simproots_\Psi$
may or may not be elements of $\simproots$, although it is true that $\simproots
\cap \Psi \subseteq \simproots_\Psi$. We say that $\Psi$ is of type $\Gamma'$ if
$\Psi$ is isomorphic to a root system of type $\Gamma'$.

The bilinear form $B$ on $\Z\Pi$ is nondegenerate. We have $B(\al, \be) \in \{2,
1, 0, -1, -2\}$ for all roots $\alpha,\beta\in \Phi$ , with  $B(\al, \be) = 2$
if and only if $\al = \be$ (\cite[\S6.6]{humphreys90}).  We say two roots $\al$
and $\beta$ are \emph{orthogonal} if $B(\al,\be)=0.$ The set $\Phi_\al$ of roots
orthogonal to a fixed root  $\al\in \allroots$ naturally form a subsystem. The
form $B$ is $W$-invariant in the sense that $B(\al, \be) = B(w(\al), w(\be))$
for all $w \in W$ and all $\al, \be \in \allroots$.  In particular, the action
of $W$ on $\Phi$ preserves orthogonality. 

\begin{rmk}
\label{rmk:complement}
It is known that when $\allroots$ has type $E_8$, $E_7$, $E_6$, or $D_n$ for $n
> 4$, for each $\al\in \Phi$ the subsystem $\Phi_\al$ has type $E_7$, $D_6$,
$A_5$, or $A_1 + D_{n-2}$, respectively (\cite[Theorem 2.5]{gx3}). \end{rmk}

The root systems of types $A$, $D$, and $E$ have the following well-known
coordinate representations.

Let $\ep_1, \ldots, \ep_n$ be the usual standard basis of the Euclidean space
$\R^n$.  The vectors $\{\ep_i - \ep_j : 1 \leq i \ne j \leq n\}$ form a root
system of type $A_{n-1}$.  The simple roots $\Pi=\{\al_1, \al_2, \ldots,
\al_{n-1}\}$ are given by $\al_i = \ep_i - \ep_{i+1}$.  A root $\ep_i - \ep_j$
is positive if $i < j$ and negative if $i > j$.  The highest root is $\ep_1 -
\ep_n$. The Weyl group is isomorphic to the symmetric group $S_n$ and acts by
permuting the basis $\ep_1,\ldots,\ep_n$.  The bilinear form $B$ is the usual
dot product. Two roots are orthogonal if and only if they have disjoint support,
where by the support of any element $v = \sum_{i=1}^n c_i\ep_i\in \R^n$ we mean
the set $\{i:c_i\neq 0\}$.

The vectors $\{\pm\ep_i \pm \ep_j : 1 \leq i < j \leq n\}$ form a root system of
type $D_n$.  The simple roots $\Pi=\{\al_1, \al_2, \ldots, \al_n\}$ are given by
$\al_i = \ep_i - \ep_{i+1}$ for $i < n$ and $\al_n = \ep_{n-1} + \ep_n$.  If $i
< j$ then the roots $\ep_i \pm \ep_j$ are positive, and the roots $-\ep_i \pm
\ep_j$ are negative.  The highest root is $\ep_1 + \ep_2$. The Weyl group has
order $2^{n-1}n!$ and acts by signed permutations of the orthonormal basis, with
the restriction that each element effects an even number of sign changes. The
bilinear form $B$ is the usual dot product.  Two roots $\al$ and $\be$ are
orthogonal if and only if either (a) $\al$ and $\be$ have disjoint support or
(b) $\al$ and $\be$ have the same support and $\al \ne \pm \be$.  

Note that the root system of type $A_{n-1}$ may be naturally identified as the
subsystem $\Phi(D_n)\cap \Z(\Pi\setminus\{\al_n\})$ of $\Phi(D_n)$, and
$W(A_{n-1})$ as the parabolic subgroup $W_I$ of the Weyl group $W=W(D_n)$ where
$I=S(D_n)\setminus\{s_n\}=\{s_1,\ldots,s_{n-1}\}$. 

The root system of type $E_8$ may be regarded as a subset of $\R^8$ as follows.
There are 112 roots of the form $\pm 2(\ep_i \pm \ep_j)$ where $1 \leq i \ne j
\leq 8$, and there are 128 roots of the form $\sum_{i = 1}^8 \pm \ep_i$, where
the signs are chosen so that the total number of $-$ is even. The simple roots
are $\Pi=\{\al_1, \al_2, \ldots, \al_8\}$, where $$ \al_1 = \ep_1 - \ep_2 -
\ep_3 - \ep_4 - \ep_5 - \ep_6 - \ep_7 + \ep_8 ,$$ $\al_2 = 2(\ep_1 + \ep_2)$,
and $\al_i = 2(\ep_{i-1} - \ep_{i-2})$ for all $3 \leq i \leq 8$. If $k$ is the
largest integer such that $\ep_k$ appears in $\al$ with nonzero coefficient $c$,
then $\al$ is positive if and only if $c > 0$. The highest root, $\theta_8$, is
$2(\ep_7 + \ep_8)$, and it decomposes as the sum \begin{equation}
\label{eq:theta8}
\theta_8=2\al_1+3\al_2+4\al_3+6\al_4+5\al_5+4\al_6+3\al_7+2\al_8. \end{equation}
The bilinear form $B$ is $1/4$ of the usual dot product.  The highest root
$\theta_8$ has the property that $B(\theta_8,\al_8)=1$ and $B(\theta_8,
\al_i)=0$ for all $1\le i\le 7$.

If $\Phi=\Phi(E_8)$ is the root system of type $E_8$ constructed above, then the
root system $\Phi(E_7)$ of type $E_7$ may be obtained from $\Phi$ as a
subsystem, namely, as the set $\Phi\cap\Z(\Pi\setminus\{\al_8\})$; the highest
root in $\Phi(E_7)$ is $2(\ep_8-\ep_7)$, and it decomposes as the sum
\begin{equation} \label{eq:theta7} \theta_7 =
2\al_1+2\al_2+3\al_3+4\al_4+3\al_5+2\al_6+\al_7. \end{equation} Similarly, the
root system $\Phi(E_6)$ of type $E_6$ can be identified as the set
$\Phi\cap\Z(\Pi\setminus\{\al_8,\al_7\})$. 

\begin{defn}
    \label{def:i-height}
    Let $\Phi$ be any root system with Dynkin diagram and let $\Pi=\{\al_i:i\in
    \Gamma\}$ be the simple roots of $\Phi$. For each root $\al=\sum_{i \in
    \Gamma} c_i \al_i$ and each $i\in \Gamma$, we define the {\it $i$-height},
    denoted $\xht_i(\al)$, to be the coefficient $c_i$.  For all $i\in \Gamma$
    and $j\in \Z_{\ge 0}$, we define 
    \[ \ihtj{i}{j}=\{\al\in \Phi_+:\xht_i(\al)=j\}. \] 
\end{defn}
\noindent
We note that in $\Phi(E_8)$, the highest root $\theta_8$ is the only root of
$8$-height 2.

\subsection{Orthogonal sets of a fixed size}
\label{sec:kroots}
We introduce the notion of $k$-roots and a natural Weyl group action on them in
this subsection.

\begin{defn}
\label{def:abs}
Let $\Phi$ be a root system with Weyl group $W$, and let $k$ be a positive
integer.
\begin{enumerate}
  \item For each root $\al\in \Phi$, we define $\abs{\al}$ to be the positive
    root in the set $\{\al,-\al\}$.

  \item We define a \emph{$k$-root} of $\Phi$ (or of $W$) to be a set of $k$
    mutually orthogonal roots in $\Phi$, and we say a $k$-root is
    \emph{positive} if all its elements are positive. We denote the set of all
    $k$-roots by $\Phi^k$ and the set of all positive $k$-roots by $\Phi_+^k$.
    More generally, for any $U\se \Phi$ we define $U^k$ and $U_+^k$ to be the
    sets of $k$-roots and positive $k$-roots whose elements all lie in $U$,
    respectively.

  \item We define 
    \[ \abs{R}=\{\abs{\al}: \al\in R\} \] 
    for each $k$-root $R\in \Phi^k$, and we define the \emph{standard action} of
    $W$ on the positive $k$-root to be the $W$-action given by 
    \begin{equation} \label{eq:stdact} 
    w\cdot R =\abs{\{w\cdot \al: \al\in R\}}=\{\abs{w\cdot \al}: \al\in R\} 
    \end{equation} 
    for all $w\in W$ and $R\in \Phi_+^k$.  
\end{enumerate}
\end{defn}

Note that Equation \eqref{eq:stdact} indeed defines a $W$-action on $\Phi_+^k$.
Also note that since orthogonal sets in $R$ are linearly independent in the
reflection representation $V$, $k$-roots cannot exist for
$k>\dim(V)=\rank(\Phi)$.  The notion of $k$-roots generalizes the notions of
$2$-roots and $n$-roots studied in \cite{gx3,gx5}, where $n=\rank(\Phi)$; see
\cite[Remark 3.7]{gx3} and \cite[Section 1]{gx5}.

\begin{lemma}
\label{lem:nroots}
Let $\Phi$ be an irreducible simply-laced root system of rank $n$, and let $W$
be the Weyl group of $\Phi$.

\begin{itemize}
\item [\rm (i)] There exists an $n$-root of $\Phi$ (i.e., the set $\Phi^n$ is
  nonempty) if and only if $\Phi$ is of type $E_7, E_8$, or $D_n$ for some even
  integer $n$. When this is the case, every set of mutually orthogonal roots of
  $\Phi$ can be extended to an $n$-root. 
\item [\rm (ii)]  If $\Phi$ has type $E_7, E_8$, or $D_n$ for $n$ even, then the
  element-wise action of $W$ on $\Phi^n$ given by  \[w\cdot
  \{\al_1,\ldots,\al_n\}=\{w\cdot \al_1,\ldots, w\cdot\al_n\}\] is transitive,
  and so is the standard action of $W$ on the positive $n$-roots $\Phi_+^n$
given by Equation \eqref{eq:stdact}. 
\end{itemize} 
\end{lemma}

\begin{rmk}
  \label{rmk:max}
  \begin{enumerate}
    \item By a {perfect matching} of the set $[n] := \{1, 2, \ldots, n\}$ for an
      even integer $n=2k$, we mean a partition of the form $ M=\{\{i_1,
      j_1\}, \ldots,  \{i_k, j_k\}\}$ where $\{i_1, j_1, \ldots, i_k,j_k\} =
      \{1, 2, \ldots, 2k\}$. 
    \item From now on, whenever we speak of ``$n$-roots'' of a root system
      $\Phi$ we will assume that $\Phi$ is a root system of rank $n$ and type
      $E_7, E_8$, or $D_n$ for $n$ even. By Lemma \ref{lem:nroots} (i), in
      these cases the $n$-roots coincide with the maximal sets (with respect
      to set inclusion) of pairwise orthogonal roots, as well as with the
      bases of the reflection representation $V$ that consist of pairwise
      orthogonal roots.  We will thus use the terms ``$n$-roots'', ``maximal
      orthogonal set'', and ``orthogonal basis'' interchangeably. 
    \item The fact that any general $k$-root can be extended to an $n$-root
      makes $n$-roots very useful, as will be evident in sections
      \ref{sec:further} and \ref{sec:45}. 
  \end{enumerate} \end{rmk}

\begin{proof}[Proof of Lemma \ref{lem:nroots}]
Part (i) is well known, and can be proved similarly to \cite[Lemma 3.2]{gx5} by
an inductive argument.  The first assertion of (ii) is \cite[Lemma 3.2]{gx5},
and the second assertion then follows immediately. 
\end{proof}

The following example on the $4$-roots in type $D_4$ turns out to be crucial for
the general theory of $n$-roots.

\begin{exa}
    \label{exa:4roots}
    The root system $\Phi=\Phi(D_4)$ of type 
    $D_4$ has exactly three positive 4-roots, namely, the sets 
    \[
        Q_0=\{\ep_{1}-\ep_2,\ \ep_1+\ep_2,\ \ep_3-\ep_4,\ \ep_3+\ep_4\},
        \quad 
        Q_1=\{\ep_{1}-\ep_3,\ \ep_1+\ep_3,\ \ep_2-\ep_4,\ \ep_2+\ep_4\},
    \]
    and 
    \[ 
        Q_2=\{\ep_{1}-\ep_4,\ \ep_1+\ep_4,\ \ep_2-\ep_3,\ \ep_2+\ep_3\}.
    \]
They correspond to the matchings $\{\{1,2\},\{3,4\}\}$,
$\{\{1,3\},\{2,4\}\}$, and $\{\{1,4\},\{2,3\}\}$, respectively.  The reflection
$s_\al$ for the root $\al=\ep_2-\ep_3$ acts on roots as the transposition
$(2,3)$ on indices, so it interchanges $Q_0$ and $Q_1$ while fixing $Q_2$ in the
standard action of the Weyl group on $\Phi_+^4$. 

Note that the 4-roots $Q_0,Q_1,Q_2$ form a partition of $\Phi_+$.  Also note
that decomposing each root as a sum of simple roots yields $$Q_0 = \{\al_1, \
\al_3, \ \al_4, \ \al_1 + 2\al_2 + \al_3 + \al_4\}, \quad Q_1 = \{\al_1 + \al_2,
\ \al_2 + \al_3, \ \al_2 + \al_4, \ \al_1 + \al_2 + \al_3 + \al_4\},$$ and $$
Q_2 = \{\al_2, \ \al_1 + \al_2 + \al_3, \ \al_1 + \al_2 + \al_4, \ \al_2 + \al_3
+ \al_4\}.$$ We will use these facts in Section \ref{sec:D4}. 
\end{exa}

\subsection{Coplanar quadruples and \texorpdfstring{$D_4$}{D4}-subsystems}
\label{sec:D4}
Let $\Phi$ be a root system of type $E_7, E_8$, or $D_{n}$ for $n$ even, so that
$n$-roots exist for $\Phi$, and let $W$ be the Weyl group of $\Phi$. In
\cite{gx5}, we showed that certain quadruples of orthogonal roots called
coplanar quadruples are key to understanding the $n$-roots of $\Phi$ and the
action of $W$ on them. These quadruples remain useful for understanding the more
general $k$-roots in this paper, so we further develop their properties in this
subsection.

\begin{defn}
    \label{def:coplanar}
    Let $\Phi$ be a root system of type $E_7, E_8$, or $D_n$ for $n$ even. We
    define a set of four distinct positive roots
    $Q=\{\be_1,\be_2,\be_3,\be_4\}\se \Phi$ to be a \emph{coplanar quadruple} if
    $\be_1+\be_2+\be_3+\be_4=2\be$ for some $\be\in \Phi$. 
\end{defn}

\begin{exa}\label{exa:quads} 
The 4-roots $Q_0,Q_1$ and $Q_2$ of type $D_{4}$ from Example \ref{exa:4roots}
all sum to twice a root, namely, twice the roots
$\beta_1=\ep_1+\ep_3,\beta_2=\ep_1+\ep_2$, and $\beta_3=\ep_1+\ep_2$,
respectively, so they are all coplanar quadruples. 

Although every orthogonal quadruple in type $D_4$ is coplanar, this is not true
in general. For example, the four orthogonal simple roots $\{\al_1, \al_3,
\al_5, \al_6\}$ in type $D_6$ do not form a coplanar quadruple. \end{exa}

Coplanar quadruples are intimately connected to $D_4$-subsystems of $\Phi$. To
make this precise, it is helpful to have the following definition. 

\begin{defn}
\label{def:D46}
Let $\Phi$ be a root system of type $E_7, E_8$, or $D_n$ for $n$ even. Let $R
\subset \allroots$ be a set of mutually orthogonal positive roots of $\Phi$. For
each integer $4\le k\leq |R|$, we define $D_k(R)$ to be the set \[D_k(R)=\{H
    \subseteq R: \abs{H}=k \text{ and }\Span(H) \cap \allroots \text{ is a
subsystem of type $D_k$ in $\Phi$}\}.\] We define $\QQ(R)$ to be the set of all
$D_4$-subsystems $\Psi$ of $\allroots$ such that $\Psi \cap R \in D_4(R)$. 
\end{defn}

The following proposition expands \cite[Proposition 3.3]{gx5} and details many
useful properties of coplanar quadruples. 

\begin{prop}\label{prop:qtfae} 
Let $\Phi$ be a root system of type $E_7$, $E_8$, or $D_n$ for $n \geq 4$ even,
let $V$ be the reflection representation of the Weyl group of $\Phi$, and let
$R\in \Phi_+^n$ be an orthogonal basis of $V$.
\begin{itemize}
  \item[{\rm (i)}] {If $\beta\in \posroots\setminus R$, then $\beta$ is
      orthogonal to all but four elements
      $Q=\{\beta_1,\beta_2,\beta_3,\beta_4\}$ of $R$. Moreover, in this case we
      have $\beta= (\pm \be_1 \pm \be_2 \pm \be_3 \pm \be_4)/2$ for suitable
      choices of sign; the set       
      \begin{equation} \label{eq:24} \Psi_Q:= Q \cup (-Q) \cup \{(\pm \be_1 \pm
      \be_2 \pm \be_3 \pm \be_4)/2\} \end{equation} is a subset of $\Phi$ that
      forms a subsystem of type $D_4$; and $\Psi_Q$ is the smallest subsystem of
      $\Phi$ containing both $Q$ and $\beta$.} 
    \item[{\rm (ii)}]{The following are equivalent for a subset $Q \subseteq R$: 
        \begin{itemize} 
          \item[{\rm (a)}]{$Q$ is a coplanar quadruple, i.e., we have $\sum_{\be
            \in Q} \be = 2 \al$ for some root $\al$;} 
          \item[{\rm (b)}]{$Q$ is the support of some root $\al \in \allroots
      \backslash (R \cup -R)\se V$ with respect to the basis $R$;} \item[{\rm
(c)}]{$Q \in D_4(R)$;} 
          \item[{\rm (d)}]{$Q = \Psi \cap R$ for some $\Psi \in
\QQ(R)$.} \end{itemize}} 

        \item[{\rm (iii)}]{The maps \[ \phi: D_4(R)\ra \QQ(R), \ Q\mapsto
              \Span(Q)\cap \Phi\] and \[ \psi: \QQ(R)\ra D_4(R), \
            \Psi\mapsto\Psi\cap R \] are mutually inverse maps. Every root
            $\alpha\in \Psi\setminus (R\cup -R)$ lies in $\Span(Q)$ for a unique
            element $Q\in D_4(R)$, namely, the support of $\alpha$.}
\end{itemize} 
\end{prop}

\begin{rmk}
\label{rmk:coplanar}
For each maximal orthogonal set of roots $R$, part (ii) of Proposition
\ref{prop:qtfae} shows that the elements of $D_4(R)$ are precisely the coplanar
quadruples in $R$, and part (iii) establishes natural bijections between the
coplanar quadruples and the $D_4$-subsystems of $\Phi$ intersecting $R$ in an
element of $D_4(R)$. Part (ii) also shows that each positive root $\alpha$ not
in $R$ naturally gives rise to a coplanar quadruple in $R$, namely, the support
of $\alpha$ with respect to $R$. 
\end{rmk}

\begin{proof}[Proof of Proposition \ref{prop:qtfae}]
Part (i) follows from \cite[Proposition 3.8]{gx5} and the first paragraph of its
proof.

To prove (ii), first note that (a) immediately implies (b). Suppose that
$\alpha$ and $Q$ satisfy the conditions of (b). Since $R$ is orthogonal, for
each root $\gamma = \sum_{\beta_i\in R} c_i\beta_i \in \Phi$ we have
$c_i=B(\gamma,\beta_i)/2$, so the support of $\gamma$ with respect to $R$
consists precisely of the elements in $R$ not orthogonal to $\gamma$. It follows from
(i) that $Q$ is a quadruple $Q=\{\beta_1, \beta_2, \beta_3, \beta_4\}$, and that
the $D_4$-subsystem $\Psi_Q$ defined in Equation \eqref{eq:24} is contained in
$\Span(Q)\cap \allroots$. Conversely, for any root
$\gamma=c_1\beta_1+c_2\beta_2+c_3\beta_3+c_4\beta_4\in \Span(Q)$, we have \[ 2=
  B(\gamma,\gamma)=2c_1^2+2c_2^2+2c_3^2+2c_4^2 \] and \[
c_i=B(\gamma,\beta_i)/2\in \{0,\pm 1, \pm (1/2)\} \] for each $i$ by Section
\ref{sec:kroots}, so either we have $c_i=1$ and $\gamma=\pm\beta_i$ for some
$i$, or we have $c_i=\pm(1/2)$ for all $i$.  It follows that $\Span(Q)\cap
\allroots=\Psi_Q$, which proves (c). The same argument also shows that if a root
$\gamma$ not in $R$ lies in $\Span(Q)$ for a quadruple $Q\se R$, then $Q$ is
necessarily the support of $\gamma$, which proves the last sentence in (iii). 

Now assume the hypothesis of (c), so that $\Psi = \Span(Q) \cap \allroots$ is a
$D_4$-subsystem of $\allroots$.  This implies that $Q \subseteq \Psi \cap R$,
and that $\Span(\Psi)$ is a 4-dimensional subspace of $V$. Since $R$ is an
orthogonal basis for $V$, we must have $|\Psi \cap R| \leq 4 = |Q|$, proving
(d).

Finally, assume the hypothesis of (d), so that $Q = \Psi \cap R$. Any orthogonal
quadruple $Q = \{\be_1, \be_2, \be_3, \be_4\}$ in a $D_4$-subsystem has the
property that $\gamma = (\be_1 + \be_2 + \be_3 + \be_4)/2$ is a root, which
proves (a) and completes the proof of (ii).

The map $\psi$ sends $\QQ(R)$ to $D_4(R)$ since (d) implies (c) in (ii), and the
proof of the fact that (c) implies (d) shows both that $\phi$ sends $D_4(R)$ to
$\QQ(R)$ and that $\psi\circ\phi$ is the identity map.  For each $\Psi\in
\QQ(R)$ and the quadruple $Q=\Psi\cap R$, the proof of the fact that (b) implies
(c) shows that any positive root $\alpha\in \Psi\setminus Q$ must satisfy
$B(\al,\beta)=\pm 1$ for all $\beta\in Q$ and have $Q$ as its support, and that
$\phi(Q)=\Span(Q)\cap \Phi=\Psi_Q=\Psi$, where the last equality holds by (i).
It follows that $\phi\circ \psi$ is the identity map on $\QQ(R)$, and therefore
$\phi$ and $\psi$ are mutual inverses. 
\end{proof}

As in \cite{gx5}, it will be useful in this paper to further categorize coplanar
quadruples into three mutually exclusive types. The definition of these types
comes from the combinatorics of perfect matchings. 

\begin{defn}
  \label{def:acn}
For any two blocks $\{a, b\}$ and $\{c, d\}$ in a perfect matching $M$ such that
$a < b$ and $c < d$, we call the pair $P = \{\{a, b\}, \{c, d\}\}$ an {\it
alignment} if $a < b < c < d$, a {\it crossing} if $a < c < b < d$, and a {\it
nesting} if $a < c < d < b$; see Figure \ref{fig:acn}. We also call each
alignment, crossing, or nesting a \emph{feature}. 
\end{defn}

It follows from Definition \ref{def:acn} that any pair of blocks in a perfect
matching forms exactly one type of feature. For example, the matchings
$\{\{1,2\},\{3,4\}\}$, $\{\{1,3\},\{2,4\}\}$, and $\{\{1,4\},\{2,3\}\}$
corresponding to the coplanar quadruples $Q_0,Q_1,Q_2$ in the root system of
type $D_4$ from Example \ref{exa:quads} form an alignment, crossing, and
nesting, respectively. Note that by Example \ref{exa:4roots}, the quadruples
$Q_0,Q_1$ and $Q_2$ can be distinguished by the numbers of simple roots in them,
because they contain $3,0,$ and $1$ simple root(s), respectively.  This
motivates the following definition, which extends the notions of alignments,
crossings, and nestings to the context of orthogonal roots.

\begin{figure}[!ht]
    \centering
    \subfloat[{alignment}]
    {
\begin{tikzpicture}
\draw (-1.8,0)--(1.8,0);
\draw[style=thick] (-1.5,0) to[out=-90, in=-180] (-1,-0.8) to[out=-0, in=-90] (-0.5,0);
\draw[style=thick] (0.5,0) to[out=-90, in=-180] (1,-0.8) to[out=0, in=-90] (1.5,0);

\node at (-1.5,0.2) {\small$a$};
\node at (-0.5,0.2) {\small$b$};
\node at (0.5,0.2) {\small$c$};
\node at (1.5,0.2) {\small$d$};
\end{tikzpicture}
}
\quad\quad\quad
\subfloat[{crossing}]
{
\begin{tikzpicture}
\draw (-1.8,0)--(1.8,0);
\draw[style=thick] (-1.5,0) to[out=-90, in=-180] (-0.5,-0.8) to[out=-0, in=-90] (0.5,0);
\draw[style=thick] (-0.5,0) to[out=-90, in=-180] (0.5,-0.8) to[out=0, in=-90] (1.5,0);

\node at (-1.5,0.2) {\small$a$};
\node at (-0.5,0.2) {\small$c$};
\node at (0.5,0.2) {\small$b$};
\node at (1.5,0.2) {\small$d$};
\end{tikzpicture}
}
\quad\quad\quad
\subfloat[{nesting}]
{
\begin{tikzpicture}
\draw (-1.8,0)--(1.8,0);
\draw[style=thick] (-1.5,0) to[out=-90, in=-180] (0,-0.8) to[out=-0, in=-90] (1.5,0); 
\draw[style=thick] (-0.5,0) to[out=-90, in=-180] (0,-0.4) to[out=0, in=-90](0.5,0);

\node at (-1.5,0.2) {\small$a$};
\node at (-0.5,0.2) {\small$c$};
\node at (0.5,0.2) {\small$d$};
\node at (1.5,0.2) {\small$b$};
\end{tikzpicture}
}
\caption{Features formed by two blocks $\{a,b\}, \{c,d\}$ in a perfect 
matching}
\label{fig:acn}
\end{figure}

\begin{defn}\label{def:features} 
    Let $Q$ be a coplanar quadruple in $\Phi$ and let $\Psi=\Span(Q)\cap \Phi$
    be the corresponding subsystem of type $D_4$. Let $\Pi_\Psi$ be the induced
    simple system of $\Psi$. We call $Q$ an \emph{alignment} if $\abs{Q\cap
    \Pi_\Psi}=3$, a \emph{crossing} if $\abs{Q\cap \Pi_\Psi}=0$, and a
    \emph{nesting} if $\abs{Q\cap \Pi_\Psi}=1$. We also call each alignment,
    crossing, or nesting a \emph{feature}. 
\end{defn}

\begin{rmk}
\label{rmk:acndef}
Definition \ref{def:features} resembles but differs from \cite[Definition
3.9]{gx5}, where the classification of coplanar quadruples into the three types
of possible features is characterized in a less concise way. However, the
definitions are equivalent: in the notation of Example \ref{exa:quads}, each
coplanar quadruple is necessarily one of $Q_0, Q_1$ and $Q_2$, and both
definitions classify $Q_0$ as an alignment, $Q_1$ as a crossing, and $Q_2$ as an
nesting.  In other words, both definitions guarantee that each coplanar
quadruple $Q$ is an alignment, crossing, or nesting if and only if the perfect
matching corresponding to it is a feature of the same type under the usual
coordinates for the $D_4$-subsystem $\Psi=\Span(Q)\cap \Phi$ associated to $Q$.
\end{rmk}

The next result discusses how distinct coplanar quadruples may interact in a
fixed $n$-root. For any integer $n>4$, we define two positive roots $\al, \be
\in \posroots$ in a root system of type $D_n$ to be a {\it couple} if any
$\gamma \in \posroots \backslash \{\al, \be\}$ is orthogonal to $\al$ if and
only if $\gamma$ is orthogonal to $\be$. In terms of coordinates, two positive
roots form a couple if and only if they are of the form $\{\ep_i + \ep_j, \
\ep_i - \ep_j\}$ for some $1 \leq i < j \leq n$.  

\begin{lemma}\label{lem:overlap} 
  Let $\Phi$ be a root system of type $E_7$, $E_8$, or $D_n$ for $n$ even. Let
  $R \in \Phi_+^n$ be a positive $n$-root, let $Q$ be a feature in $R$, and let
  $\Psi=\Span(Q)\cap \Phi$ be the associated $D_4$-subsystem of $\Phi$. 

  \begin{enumerate}
    \item If $Q'$ is feature in $R$ such that $Q\neq Q'$ and $Q \cap Q'\ne
      \emptyset$, then the set $H:=Q\cup Q'$ lies in $D_6(R)$. In other words,
      we have $|H|=6$, and $\Xi := \Span(H) \cap \allroots$ is a subsystem of
      type $D_6.$ Moreover, in this case the sets $Q$, $Q'$, $Q \cap Q'$, and $Q
      \cup Q'$ all consist of couples in $\Xi$.

    \item If $\Phi$ has type $E_8$, then the set 
      \[
        \Psi^\perp:=\{\al\in \Phi:B(\al,\be)=0\;\text{for all }\be\in \Psi\}
      \]
      is a $D_4$-subsystem of $\Phi$ and is the unique $D_4$-subsystem
      consisting of roots orthogonal to $\Psi$. Furthermore, the set $R\setminus
      Q$ is a feature in $\Psi$.
  \end{enumerate}
\end{lemma}

\begin{proof} 
Part (i) is a restatement of \cite[Proposition 3.20]{gx5}. 

Suppose $\Phi$ has type $E_8$. If we write $Q=\{\be_1,\be_2,\be_3,\be_4\}$, then
the proof of \cite[Lemma 3.17]{gx5} shows that the set of roots $\Phi$
orthogonal to $\be_1,\be_2,\be_3$ forms a root subsystem of type $A_1+D_4$ in
$\Phi$, where the component of type $A_1$ is precisely $\{\pm \be_4\}$. It
follows that $\Psi^\perp$ is the remaining component that forms a
$D_4$-subsystem of $\Phi$, which proves the first sentence in (ii). The second
sentence in (ii) follows from \cite[Remark 3.16]{gx5} and \cite[Lemma
3.18]{gx5}, which show that the set of features in any $R\in \Omega_{\Phi_+}$
forms a Steiner quadruple system and is closed under the complement map
$Q\mapsto R\setminus Q$. \end{proof}

\section{Generalized Rothe diagrams and quasiparabolic sets}
\label{sec:rqs}

In Section \ref{sec:gendef}, we introduce the definitions of generalized Rothe
diagrams and several related objects, including generalized quantum Hafnians
(Definition \ref{def:residues}). We also recall Rains and Vazirani's notion of
quasiparabolic sets, which provides a common framework for the sets of
$k$-roots that we will study in the remaining parts of the paper. In Section
\ref{sec:motexa}, we give details of two of the motivating examples of
generalized Rothe diagrams: one based on perfect matchings of a set of size $2k$
(Theorem \ref{thm:hafpfaf}), and one based on permutations of $k$ objects
(Theorem \ref{thm:perdet}). In Section \ref{sec:reading}, we recall Reading's 
notion of poset congruences and show how it can be used to construct new examples 
of quasiparabolic sets.

\subsection{Key definitions}
\label{sec:gendef}
The following two definitions are central to this paper. They make sense for all
root systems, and they are new to the best of our knowledge.

\begin{defn}
\label{defn:dominate}
Let $\Phi$ be an arbitrary root system and let $\al,\be\in \Phi_+$. We say that
$\al$ \emph{dominates} $\be$ if $s_\be(\al)$ is a
positive root.  
\end{defn}

\begin{defn}\label{def:residues} 
  Let $\allroots$ be an arbitrary root system, and let $U\se \Phi_{+}$. 
  \begin{enumerate}
\item We define $\Omega_U$ to be the set of all $\kp$-roots of $U$, where $\kp$
  is the maximum cardinality of a pairwise orthogonal subset of $U$.

\item For any orthogonal set $R \in \Omega_U$, we say that a positive root $\al\in
  \Phi_+$ \emph{dominates $R$}, or that $\al$ is \emph{$R$-dominating}, if $\al$
  dominates every element of $R$. We define the {\it generalized Rothe diagram of
  $R$ with respect to $U$} to be the set \[\res_U(R) = \{\al\in U: \al
  \text{\;dominates\;} R\}= \{\al\in U: s_\be(\al) > 0 \text{\ for\ all\ } \be
\in R\}.\] We define the \emph{level of $R$ with respect to $U$} to be
$\rho_U(R):=\abs{\res_U(R)}$.

\item We associate to each element $R = \{\be_1, \be_2, \ldots, \be_k\} \in
  \Omega_U$ the monomial $\prod_{\be \in R} x_\be$, where $\{x_\be : \be \in
  U\}$ is a set of commuting indeterminates.  The {\it generalized quantum
  Hafnian of $U$}, $\qhf(U) = \qhf(U, q)$, is the element of $\Z[U][q]$ given by
  $$ \sum_{R \in \Omega_U} q^{\lht(R)} \prod_{\be_i \in R} x_{\be_i} .$$ The
  specializations $\qpf(U) := \qhf(U, -q)$, $\hf(U) := \qhf(U, 1)$, and $\pf(U)
  := \qhf(U, -1)$ are called the {\it generalized quantum Pfaffian}, the {\it
  generalized Hafnian}, and the {\it generalized Pfaffian} of $U$, respectively.

\item We define the {\it Poincar\'e polynomial} of $U$ to be the polynomial
  \[PS_U(q):= \sum_{R \in \Omega_U} q^{\lht(R)} \in \Z[q]  \] obtained from
  $\qhf(U)$ by specializing all the indeterminates $x_\be$ to $1$.  If $PS_U(q)$
  has the form $\prod_{d \in D} [d]_q$ for some multiset $D$ of elements from
  $\N\backslash\{0, 1\}$, where $[d]_q$ is the quantum integer $$ [d]_q =
  \frac{q^d - 1}{q - 1} = 1 + q + q^2 + \ldots + q^{d-1},$$ then we call the
  elements of $D$ the {\it degrees} of $U$.  \end{enumerate} \end{defn}

\begin{rmk}
\label{rmk:negate}
\begin{enumerate}
\item 
Since $s_\gamma(\gamma)=-\gamma<0$ for any root $\gamma\in \Phi$, no positive root
dominates itself, so the set $\res_U(R)$ cannot contain any element of $R$ in
the setting of Definition \ref{def:residues}. 

\item Factorization of a polynomial in $\Z[q]$ into quantum integers is unique
up to order when possible, so the multiset $D$ in Definition \ref{def:residues}
(iv) is well-defined. \end{enumerate} \end{rmk}

\begin{exa}\label{exa:residues} 
Let $\allroots$ be the root system of type $D_4$, whose highest root is
$\theta=\alpha_1+2\al_2+\al_3+\al_4$, and define $U = \posroots$. There are
three maximal orthogonal subsets of positive roots in $U$, namely, 
\begin{align*}
Q_0 &= \{\al_1, \ \al_3, \ \al_4, \ \al_1 + 2\al_2 + \al_3 + \al_4\},\\
Q_1 &= \{\al_1 + \al_2, \ \al_2 + \al_3, \ \al_2 + \al_4, \ \al_1 + \al_2 + \al_3 + \al_4\}, \quad \text{and}\\
Q_2 &= \{\al_2, \ \al_1 + \al_2 + \al_3, \ \al_1 + \al_2 + \al_4, \ \al_2 + \al_3 + \al_4\}.
\end{align*}
Direct computation shows that $$ \res_{U}(Q_0) = \emptyset, \quad \res_U(Q_1) =
\{\theta\} = \{\ep_1 + \ep_2\}, \quad \text{and} \quad \res_U(Q_2) = \{\theta -
\alpha_2, \ \theta\} = \{\ep_1 + \ep_2, \ \ep_1 + \ep_3\},$$ from which it
follows that $\rho(Q_i)=\rho_U(Q_i)= i$ for all $0\le i\le 2$.

Using the coordinate system for roots of type $D_4$, and writing $u_{ij}$ as
shorthand for $x_{\ep_i + \ep_j}x_{\ep_i - \ep_j}$, we have
    \begin{align*} 
      \qhf(U)&= u_{12}u_{34} + qu_{13}u_{24} + q^2
        u_{14}u_{23},\\ \qpf(U) &= u_{12}u_{34} - qu_{13}u_{24} + q^2
        u_{14}u_{23},\\ \hf(U)&= u_{12}u_{34} + u_{13}u_{24} +
        u_{14}u_{23}, \quad \text{\rm and} \\ \pf(U) &= u_{12}u_{34} -
        u_{13}u_{24} + u_{14}u_{23}.
    \end{align*} 
It follows that $PS_\posroots(q) = 1 + q + q^2$, so that $3$ is the only degree
of $U$.

Note that $Q_0, Q_1,$ and $Q_2$ are exactly the three coplanar
quadruples partitioning $\posroots$ from Example \ref{exa:quads}, with $Q_0$
being an alignment, $Q_1$ a crossing, and $Q_2$ a nesting. The highest root
$\theta$ and the second highest root $\theta - \al_2$ are the only roots that
may appear in $\res_U(R)$ for any $R\in \Omega_U$. Note also that if a root $\al
\in \{\theta, \theta - \al_2\}$ dominates $Q_i$ for some $i$, then
$\al$ has full support with respect to the orthogonal basis $Q_i$.
\end{exa}

The next definition considers an important special case of Definition
\ref{def:residues} (ii), where the reference set $U$ in Definition
\ref{def:residues} no longer needs to be specified independently. Note that by
Section \ref{subsec:roots}, if $R$ is a $k$-root of a root system $\Phi$, then
the set $\Psi:=\Span(R)\cap\Phi$ is a subsystem of $\Phi$, and the positive
roots $\Phi_+$ of $\Phi$ induce a positive system
$\Psi_+=\Psi\cap\Phi_+=\Span(R)\cap\Phi_+$ of $\Psi$. 

\begin{defn}\label{def:internal_residues} 
Let $\allroots$ be an arbitrary root system, let $R\se\Phi_+$ be a $k$-root of
$\Phi$, and let $\Psi_R=\Span(R)\cap \Phi$ be the subsystem of $\Phi$ induced by
$R$. We define the {\it generalized Rothe diagram of $R$} to be the set \[
  \res(R) = \res_{\Psi_+}(R), \] where $\Psi_+=\Span(R)\cap\Phi_+$ is the
positive system of $\Psi$ induced by $\Phi_+$. We define the \emph{level}, $\rho(R)$,
of $R$ to be $\abs{\res(R)}$. \end{defn}

\begin{rmk}
\label{rmk:int_res}
\begin{enumerate}
\item It follows from Definitions \ref{def:residues} (ii) and
  \ref{def:internal_residues} that for any $k$-root $R$ of a root system $\Phi$,
  we have $\res (R) = \res_{\Phi_+}(R)\cap\Span(R)$. In particular, if $R$ is an
  $n$-root where $n=\rank(\Phi)$, then $\res(R)=\res_{\Phi_+}(R)$ and
  $\rho(R)=\rho_{\Phi_+}(R)$. 

\item In the setting of Example \ref{exa:residues}, we have
  $\res(Q_i)=\res_{\Phi_+}(Q_i)=\res_{U}(Q_i)$ and
  $\rho(Q_i)=\rho_{\Phi_+}(Q_i)=\rho_U(Q_i)$ for each $i$. Since the calculation
  of  $\res(Q_i)$ in this example depends only on the choice of a simple
  system for a type $D_4$ root system, it follows that whenever $Q$ is a feature
  in a root system of type $E_7, E_8$ or $D_n$ for $n$ even, we have $\rho(Q) =
  0$ if $Q$ is an alignment,  $\rho(Q)=1$ if $Q$ is a crossing, and $\rho(Q)=2$
if $Q$ is a nesting. \end{enumerate} \end{rmk}

The rest of the paper will study sets of the form $\Omega_U$ that are
quasiparabolic sets under their associated level functions, in the following
sense.

\begin{defn}
\cite[Section 2, Section 5]{rains13}
\label{def:qpset}
Let $(W,S)$ be a Coxeter system with set of reflections $T$. A {\it scaled
$W$-set} is a pair $(X, \lambda)$ where $X$ is a $W$-set and $\lambda : X
\rightarrow \Z$ is a function satisfying $|\lambda(sx) - \lambda(x)| \leq 1$ for
all $s \in S$.  An element $x \in X$ is {\it $W$-minimal}  if $\lambda(sx) \geq
\lambda(x)$ for all $s\in S$ and is {\it $W$-maximal} if $\lambda(sx) \leq
\lambda(x)$ for all $s \in S$.

A {\it quasiparabolic set for $W$}, or \emph{quasiparabolic $W$-set}, is a
scaled $W$-set $X$ satisfying the following two properties:
\begin{itemize}[leftmargin=3.2em] \item[(QP1)]{for any $r \in T$ and $x \in X$,
  if $\lambda(rx) = \lambda(x)$, then $rx = x$;} \item[(QP2)]{for any $r \in T$,
$x \in X$, and $s \in S$, if $\lambda(rx) > \lambda(x)$ and $\lambda(srx) <
\lambda(sx)$, then $rx = sx$.} \end{itemize} For a quasiparabolic set $X$, we
define the \emph{quasiparabolic order} on $X$, $\le$, to be the weakest
partial order such that $x \le rx$ whenever $x \in X$, $r \in T$, and
$\lambda(x) \leq \lambda(rx)$. 
\end{defn}

We note that Rains and Vazirani call $\lambda(x)$ the {\it height} of $x$, call
$\le$ the \emph{Bruhat order}, and denote $\le$ by $\le_Q$, but we
have chosen slightly different terminology and notation because of the potential for
confusion in the context of this paper. For every quasiparabolic set $X$
appearing in this paper, the quasiparabolic order on $X$ will be the only
partial order that we discuss, so the meaning of $\le$ will be clear from
context.

For any quasiparabolic $W$-set $(X,\lambda)$, it follows from \cite[Corollary
2.10]{rains13} that each $W$-orbit in $X$ contains at most one $W$-minimal
element, so if $X$ is a finite and transitive $W$-set then $X$ has a unique
$W$-minimal element, $x_0$. Moreover, when this is the case, it follows from
\cite[Corollary 2.13]{rains13} that for each element $x\in X$, the minimum length
of the elements $w\in W$ such that $w(x_0)=x$ equals $\lambda(x)-\lambda(x_0)$.  

We may associate a quantum Hafnian and a Poincar\'e polynomial to a general
quasiparabolic set by defining $\qhf(X) = \sum_{x \in X} q^{\lambda(x)} x$ and
$PS_{X}(q)= \sum_{x \in X}q^{\lambda(x)}.$ However, we will not require this
level of generality, because our focus for the rest of the paper will be on
quasiparabolic sets that arise from sets of orthogonal roots via the following
construction.

\begin{defn}\label{def:qptriples} 
Let $W$ be a finite Weyl group with root system $\allroots$ and generating set
$S$, let $I \se S$, and let $U \se \posroots$.  We say that $(W, I, U)$ is a
\emph{quasiparabolic triple}, or {\it QP-triple}, if the set $\omu{U}$ is a
quasiparabolic set for $W_I$ under the restriction of the standard action
(Definition \ref{def:abs}) and level function $\lambda(R) := \rho_U(R)$. A
QP-triple $(W,I,U)$ is \emph{transitive} if $\omu{U}$ is a transitive
quasiparabolic $W_I$-set. \end{defn}

\subsection{Motivating examples}
\label{sec:motexa}
We now present the two motivating examples of Definition \ref{def:residues}: one
based on perfect matchings, and the other based on permutations. The following
definition will be helpful in the context of the first example.

\begin{defn}\label{def:sigmatau}
Let $M=\{\{i_1,j_1\}, \ldots, \{i_k,j_k\}\}$ be a perfect matching of the set
$[2k]$, where we have $k \geq 2$, $i_1<\ldots <i_k$, and $i_t<j_t$ for all $1\le
t\le k$.  Following \cite[\S2]{jing14}, we define $\sigma_M\in S_{2k}$ to be the
permutation whose two-line notation is given by 
\[ \sigma_M = \begin{bmatrix} 1
& 2 & \ldots & 2k-1 & 2k \\ i_1& j_1 &  \ldots  & i_k  & j_k \end{bmatrix}. \]
We define $\tau_M\in S_{2k}$ to be the fixed-point-free involution whose
disjoint cycle notation is given by 
$$ \tau_M = (i_1, j_1)(i_2, j_2) \ldots
(i_k, j_k) .$$ 
Let 
$$ I_M = \{(i, j): 1 \leq i < j \leq 2k \text{\ and\ } \sigma_M(i) > \sigma_M(j)\} $$ 
be the set of inversions of $\sigma_M$, and define 
$$ D_M = \{(i, j): 1 \leq i < j \leq 2k \text{\ and\ } \tau_M(i) > j \text{\
and\ } \tau_M(j) > i\} .$$ 
\end{defn}

\begin{lemma}\label{lem:sigmatau}
Maintain the notation of Definition \ref{def:sigmatau}. There is a bijection
$\psi_M: I_M\ra D_M$ given by 
\[\psi_M((i, j)) = (\tau_M \sigma_M(i), \sigma_M(j)). \] 
\end{lemma}

\begin{proof}
For any integer $1\le e\le k$, the definitions of $\sigma_M$ and $I_M$ imply
that $(2e-1,2e)\notin I_M$, and the definitions of $\tau_M$ and $D_M$ imply that
$(\sigma_M(2e-1),\sigma_M(2e))\notin D_M$.  It follows that every pair $(i,j)\in
I_M$ corresponds to two distinct integers $1\le e<f\le k$ such that $i\in
\{2e-1,2e\}$ and $j\in \{2f-1,2f\}$, and every pair $(i,j)\in I_M$ consists of
elements from two blocks $\{\sigma_M(2e-1),\sigma_M(2e)\}$ and
$\{\sigma_M(2f-1), \sigma_M(2f)\}$ for such two distinct integers $1\le e<f\le
k$.

Suppose we have $1\le e<f\le k$, and let $a := \sigma_M(2e-1)$, $b :=
\sigma_M(2e)$, $c := \sigma_M(2f-1)$ and $d := \sigma_M(2f)$. It follows from
the definition of $\sigma_M$ that we have $a < c < d$ and $a < b$.  There are
three possibilities for the relative order of $a,b,c,$ and $d$: $a < c < d < b$,
$a < c < b < d$, and $a < b < c < d$.  These correspond respectively to
nestings, crossings, and alignments of the matching $M$. 

Of the three cases just described, the first case gives two inversions of
$\sigma_M$, namely $(2e, 2f-1)$ and $(2e, 2f)$, and two elements of $D_M$,
namely $(a, d)$ and $(a, c)$.  The second case gives one inversion of
$\sigma_M$, namely $(2e, 2f)$, and one element of $D_M$, namely $(a, c)$.  The
third case does not give rise to any inversions of $\sigma_M$ or any elements of
$D_M$. The map $\psi_M$ matches the inversions with the elements of $D_M$ in
each case, which completes the proof. \end{proof}

\begin{theorem}\label{thm:hafpfaf}
Let $(W, S)$ be a Weyl group of type $D_n$, where $n = 2k \geq 4$ is even, and
let $I = S \backslash \{s_{2k}\}$, so that $W_I \cong S_{2k}$. Define $U$ to be
the set of the $\binom{2k}{2}$ positive roots given by $$ U = \{\ep_i + \ep_j :
1 \leq i < j \leq 2k\} ,$$ and write $u_{ij} = u_{ji} := \ep_i + \ep_j$ for all
$1\le i<j\le 2k$. Maintain the notation of Definition \ref{def:sigmatau}, and
let $\mathcal{M}$ be the set of perfect matchings of $[2k]$. 

\begin{itemize}
\item[{\rm (i)}]{The triple $(W, I, U)$ is a QP-triple.}
\item[{\rm (ii)}]{The set $\Omega_U$ has $(2k-1)!!$ elements, each of which has
     the form $$ R(M) = \{u_{i_1 j_1}, u_{i_2 j_2}, \ldots, u_{i_k j_k}\} $$ for
     some perfect matching $M\in \mathcal{M}$.}
\item[{\rm (iii)}]{For each $M\in \mathcal{M}$, we have
\[\res_U(R(M))=\{u_{ij}: 1\le i<j\le 2k, \tau_M(i) > j, \text{and } \tau_M(j) >
i\}.\]
}
\item[{\rm (iv)}]{For each $M\in \mathcal{M}$, the set $$ \{ (i, j) : j < i
     \text{\ and\ } u_{ji} \in \res_U(R(M)) \} = \{(i, j) : (j, i) \in D_M\} $$
     is the Rothe diagram $\hat{D}_{\tt FPF}(\tau_M)$ of the fixed-point-free
     involution $\tau_M$ in the sense of Hamaker--Marberg--Pawlowski
     \cite[(3.3)]{hamaker18}.}
\item[{\rm (v)}]{For each $M\in \mathcal{M}$, we have $$ \lht_U(R(M)) =
     \ell(\sigma_M) = (\ell(\tau_M)-k)/2 ,$$ where $\ell$ denotes the length of
     a permutation.}
\item[{\rm (vi)}]{The generalized quantum Hafnian $\qhf(U)$ is the $q$-Hafnian
     in the sense of Jing--Zhang \cite[Proposition 3.6]{jing16}, and the
     generalized quantum Pfaffian $\qpf(U)$ is the $q$-Pfaffian in the sense of
     Strickland \cite{strickland96} and Jing--Zhang \cite[Proposition
     3.4]{jing16}.}
\item[{\rm (vii)}]{The generalized Hafnian $\hf(U)$ is the Hafnian of a $2k
    \times 2k$ symmetric matrix $(a_{ij})$ satisfying $a_{ij} = u_{ij}$ for all
    $1 \leq i < j \leq 2k$, and the generalized Pfaffian $\pf(U)$ is the
    Pfaffian of a $2k \times 2k$ skew-symmetric matrix $(a_{ij})$ satisfying
  $a_{ij} = u_{ij}$ for all $1 \leq i < j \leq 2k$.}
\item[{\rm (viii)}]{The Poincar\'e polynomial of $U$ is $\prod_{i = 2}^k
     [2i-1]_q$, and its degrees are $3, 5, 7, \ldots, 2k-1$.}
\end{itemize}
\end{theorem}

\begin{proof} 
Part (ii) follows from the observation that two elements of $U$ are orthogonal
if and only if they have disjoint support with respect to the $\ep$-basis. 

If $\be, u_{ij} \in U$ satisfy $s_\be(u_{ij}) < 0$, then we must either have
$\be = u_{aj}$ for some $a \leq i$, or $\be = u_{ib}$ for some $b \leq j$. It
follows that $u_{ij}$ dominates $R(M)$ if and only if we have both
$\tau_M(i) > j$ and $\tau_M(j) > i$, proving (iii).

The set $$ \{ (j, i) : j < i, \tau(i) > j \text{\ and\ } \tau(j) > i \} $$ is
the Rothe diagram $\hat{D}_{\tt FPF}(\tau)$ of the fixed-point-free involution
$\tau$ as defined by Hamaker--Marberg--Pawlowski \cite[(3.3)]{hamaker18}.  Part
(iv) follows by applying (iii) to this definition.

It follows from \cite[Proposition 3.6]{hamaker18} that the level, $\lht_U(R(M))
= |\hat{D}_{\tt FPF}(\tau_M)|$, of $R$ is equal to $(\ell(\tau_M)-k)/2$.  Lemma
\ref{lem:sigmatau} implies that $\lht_U(R(M)) = \ell(\sigma_M$), and this
completes the proof of (v).

Rains--Vazirani \cite[\S4]{rains13} proved that the function $(\ell(\tau_M) -
k)/2$ is a quasiparabolic level function for the set of fixed-point-free
involutions of $S_{2k}$, regarded as an $S_{2k}$-set in the natural way.
Combining this result with part (v) proves (i).

It follows from (v) that we have $$ \qhf(U) = \sum_M q^{\ell(\sigma_M)} u_{i_1
j_1} u_{i_2 j_2} \ldots u_{i_k j_k} ,$$ where the sum is over all perfect
matchings of $[2k]$. This agrees with Jing--Zhang's construction of the quantum
Hafnian from \cite[Proposition 3.6]{jing16}. All the assertions of (vi) and (vii)
now follow from the relevant definitions.

Part (viii) is a standard result that has been rediscovered several times; see
\cite[Proposition 7.1]{gx5}. \end{proof}

\begin{exa}\label{exa:qhf}
In the case $n=6$ of Theorem \ref{thm:hafpfaf}, we have 
  \begin{align*} 
    \qhf(U) =& \hphantom{+} u_{12}u_{34}u_{56} + qu_{13}u_{24}u_{56} +
      qu_{12}u_{35}u_{46} + q^2 u_{13}u_{25}u_{46} + q^2 u_{14}u_{23}u_{56}\\
      &+ q^2 u_{12}u_{36}u_{45} + q^3 u_{14}u_{25}u_{36} + q^3
      u_{15}u_{23}u_{46} + q^3 u_{13}u_{45}u_{26} + q^4 u_{16}u_{23}u_{45}\\
      &+ q^4 u_{15}u_{24}u_{36} + q^4 u_{14}u_{35}u_{26} + q^5
      u_{15}u_{26}u_{34} + q^5 u_{16}u_{24}u_{35} + q^6 u_{16}u_{25}u_{34}.
  \end{align*}
\end{exa}
  
\begin{rmk}\label{rmk:resprod}
In certain cases, the product of the elements in $\res_U(R(M))$ associated to a
perfect matching $M$, when regarded as a polynomial in the $\ep_i$, has an
interpretation in terms of Schubert polynomials; see \cite[Theorem
1.3]{hamaker18}. \end{rmk}

\begin{theorem}\label{thm:perdet} 
Let $(W, S)$ be a Weyl group of type $D_n$, where $n = 2k \geq 4$ is even, and
let $I = S \backslash \{s_k, s_{2k}\}$, so that $W_I \cong S_k \times S_k$.
Define $U$ to be the set of $k^2$ positive roots given by $$ U = \{\ep_i +
\ep_{j+k} : i,j\in [k]\} ,$$ write $v_{ij} := \ep_i + \ep_{j+k}$ for all $1\le
i, j\le k$. 

\begin{itemize}
\item[{\rm (i)}]{The triple $(W, I, U)$ is a QP-triple.}
\item[{\rm (ii)}]{The set $\Omega_U$ has $k!$ elements, each of which has the
     form $$ R(\pi) := \{v_{1\pi(1)}, v_{2 \pi(2)}, \ldots, v_{k\pi(k)} \}
     $$ for some permutation $\pi \in S_k$.}
\item[{\rm (iii)}]{For each $\pi\in S_k$, we have 
\[
\res_U(R(\pi))=\{v_{ij}: i, j \in [k], \pi(i)>j, \text{and } \pi\inverse(j)>i\}.
\]}
\item[{\rm (iv)}]{The set $$ \{ (i,j) : v_{ij} \in \res_U(R(\pi)) \} $$ is the
     Rothe diagram $D(\pi)$ of the permutation $\pi$.}
\item[{\rm (v)}]{For each permutation $\pi \in S_k$, we have $$ \lht_U(R(\pi)) =
     \ell(\pi) ,$$ where $\ell$ denotes the length of a permutation.}
\item[{\rm (vi)}]{The generalized quantum Hafnian $\qhf(U)$ is the $q$-permanent
     of the matrix $A = (v_{ij})$ in the sense of \cite[(2.7)]{jing14}, and the
     generalized quantum Pfaffian $\qpf(U)$ is the quantum row permanent
     $\per_q(A)$ of $A$ in the sense of \cite[(3.18)]{jing16}.}
\item[{\rm (vii)}]{The generalized Hafnian $\hf(U)$ and the generalized Pfaffian
  $\pf(U)$ are the permanent $\per(A)$, and the determinant $\det(A)$,
respectively.}
\item[{\rm (viii)}]{The Poincar\'e polynomial of $U$ is $\prod_{i = 2}^k [i]_q$,
     and its degrees are $2, 3, 4, ..., k$.}
\end{itemize}
\end{theorem}

\begin{proof} 
Parts (ii) and (iii) follow by restriction from Theorem \ref{thm:hafpfaf} (ii)
and (iii), with each $k$-root $R(\pi)$ defined in (ii) corresponding to the
perfect matching $M=\{\{i,\pi(i)+k\}: i\in [k]\}$.  Part (iv) then follows from
the definition of the Rothe diagram $D(\pi)$ \cite[(2.7)]{hamaker18}. Part (v),
which is well known, follows from the fact that there is a bijection between the
inversions of $\pi$ and the Rothe diagram of $\pi$ given by $(i, j) \mapsto (i,
\pi(j))$. 

If we identify $\Omega_U$ with $S_k$ and $W_I$ with $S_k\times S_k$, then the
$W_I$-action on $\Omega_U$ is given by $$ (w_1, w_2) \cdot w = w_1 w (w_2)^{-1}
.$$ By \cite[Theorem 3.1]{rains13}, this makes $S_k$ into a quasiparabolic set
with height function $\ell$, proving (i).

Parts (vi) and (vii) follow from the relevant definitions, and (viii) follows
from the well-known identity $\sum_{\pi \in S_k} q^{\ell(\pi)} = [k]_q!$.
\end{proof}

\begin{rmk}\label{rmk:biset}
\begin{enumerate}
  \item
  It is sometimes helpful to identify $U$ as the positions on a $k\times k$
  board and $\Omega_U$ with configurations of $k$ nonattacking rooks on the
  board, with each root $v_{ij}=\ep_i\pm \ep_{j+k}$ corresponding to the square 
  $(i, j)$ in the $i$-th row from the top and $j$-th column from the left.
  If we place a dot at
  each position corresponding to a root in an element $R\in \Omega_U$, then a
  root dominates a $k$-root if and only if it is of the form $\ep_a + \ep_{b+k}$
  where the position $(a, b)$ is both vacant and not attacked either from the left 
  or from above by the positions in $R$. For example, 
  Figure \ref{fig:og_rothe} depicts the case where $R=\{v_{12}, v_{24}, v_{31},
  v_{43}\}$, with the elements of $R$ marked by dots and
  the $R$-dominating roots by stars. 

\item Instead of regarding $\Omega_U$ as an $S_k \times S_k$-set, we may regard
  it as an $(S_k, S_k)$-biset, as in the proof of part (v) of Theorem \ref{thm:perdet}. 
  From this perspective, $\Omega_U$ has a left $S_k$-action permuting the rows, and a
  right $S_k$-action permuting the columns. The actions commute with each other,
and $\Omega_U$ is a quasiparabolic $S_k$-set under each action. \end{enumerate}
\end{rmk}

\subsection{Poset congruences}\label{sec:reading}
We recall the notion of poset congruence in the sense of Reading
\cite{reading02}. The main result of this subsection is Theorem
\ref{thm:quotient}, which proves that the equivalence classes of a
$W$-invariant poset congruence on a quasiparabolic $W$-set form another
quasiparabolic $W$-set in a natural way.

\begin{defn}\label{def:reading}
Let $(P, \leq)$ be a poset. A {\it poset congruence} on $P$ is an equivalence
relation on $P$ satisifying 
the following properties.
\begin{itemize}
\item[(i)]{The $\sim$-equivalence class $[x]$ of any $x \in P$ is an interval in $P$,
which we denote by $[x^-, x^+]$.}
\item[(ii)]{The function $f^- : P \rightarrow P$ given by $f^-(x)=x^-$ is a morphism
of posets.}
\item[(iii)]{The function $f^+ : P \rightarrow P$ given by $f^+(x)=x^+$ is a morphism
of posets.}
\end{itemize}
We denote the equivalence classes of $\sim$ by $P/\!\sim$.
For any function $\lambda: P\ra \Z$, we define $\blm:P/\!\sim\, \ra \Z$ to be
the function given by $\blm([x])=\lambda(f^-(x))$
for all $x\in P$. If $G$ is a group that acts on the underlying set of $P$, then
we say $\sim$ is \emph{$G$-invariant} if $gx\sim gy$ whenever we have $g\in G$ and
$x\sim y$ (equivalently, if the formula $g[x]=[gx]$ gives a
well-defined action of $G$ on $P/\!\sim$). 
\end{defn}

\begin{lemma}\label{lem:quotient}
Let $W$ be a Coxeter group, let $(X, \leq)$ be a quasiparabolic $W$-set, and let $\sim$ 
be a $W$-invariant poset congruence on $X$. If $\blm([x]) < \blm([rx])$ for
some $x\in X$ and some reflection $r\in W$, then we have the following:
\begin{itemize}
\item[{\rm (i)}]{$x < rx$;}
\item[{\rm (ii)}]{$\lambda(x) < \lambda(rx)$;}
\item[{\rm (iii)}]{$x^- < (rx)^- \leq r(x^-)$.}
\end{itemize}
\end{lemma}

\begin{proof}
The elements $x$ and $rx$ are comparable in $X$ because $X$ is quasiparabolic. We 
cannot have $x=rx$ because $\blm([x]) \ne \blm([rx])$, and we cannot have 
$x > rx$ because this would imply $x^- \geq (rx)^-$ by the definition of poset 
congruence, which in turn would imply that $\blm([x]) \geq \blm([rx])$. 
This proves (i), and (ii) follows from the definition of $\leq$. 

The first inequality of (iii) follows by combining (i) with the definition of 
poset congruence and the fact that $x$ and $rx$ are not $\sim$-equivalent. The 
$W$-invariance of $\sim$ implies that $rx \sim r(x^-)$, and the second inequality 
of (iii) follows from the fact that $(rx)^-$ is the minimal element in the congruence 
class $[rx]=[r(x^-)]$.
\end{proof}

The next theorem, which is new to the best of our knowledge, will turn out to be
very useful. 

\begin{theorem}\label{thm:quotient}
Let $W$ be a Coxeter group, let $(X, \leq)$ be a quasiparabolic $W$-set with level
function $\lambda$, and let $\sim$ be a $W$-invariant poset congruence on $X$. 
The set of equivalence classes $X/\!\sim$ forms a quasiparabolic $W$-set with level 
function $\blm([x]) := \lambda(x^-)$.
\end{theorem}

\begin{proof}
To prove that $X/\!\sim$ is a scaled $W$-set, let $x \in X$ and $s \in S$, and
suppose for a contradiction that $\blm(s[x]) > \blm([x]) + 1$. This would imply
that $\lambda((sx)^-)=\blm([sx])> \lambda(x^-)+1$ by the definition of $\blm$,
but the assumption that $X$ is a scaled $W$-set and Lemma \ref{lem:quotient}
(iii) imply that  $(sx)^-\le s(x^-)$ and therefore $\lambda((sx)^-) \leq
\lambda(s(x^-))\leq \lambda(x^-)+1$, which is a contradiction. It follows that
$\blm(s[x])\le \blm([x])+1$. Replacing $x$ with $y=sx$ then shows that
$\blm([x])=\blm(s[y])\le \blm([y])+1=\blm(s[x])+1$ and therefore $\blm(s[x])\ge
\blm([x])-1$. It follows that $X/\!\sim$ is a scaled $W$-set.

To prove (QP1), let $x \in X$ and $r \in T$, and suppose that
$\blm(r[x])=\blm([rx])=\blm([x])$. If $[rx] \ne [x]$, then we may assume without
loss of generality that $\blm([x]) < \blm([rx])$ by replacing $x$ by $rx$ if
necessary. Lemma \ref{lem:quotient} (iii) now implies that $x^- < (rx)^- \leq
r(x^-).$ Because $X$ is quasiparabolic, this implies that $\lambda((rx)^-) >
\lambda(x^-)$, which contradicts the hypothesis $\blm([rx])=\blm([x])$. We must
therefore have $[rx]=[x]$, which proves that $X/\!\sim$ satisfies (QP1).

To prove (QP2), let $x \in X$, $r \in T$, and $s \in S$, and suppose that
$\blm([rx]) > \blm([x])$ and $\blm([srx]) < \blm([sx])$. Lemma
\ref{lem:quotient} (ii) then implies that $\lambda(rx) > \lambda(x)$ and
$\lambda(srx) < \lambda(sx)$. Condition (QP2) of $(X, \leq)$ then proves that
$rx=sx$, which in turn implies that $[rx]=[sx]$, proving (QP2) for $X/\!\sim$.
\end{proof}

We will call each set of the form $X/\!\sim$ in the setting of Theorem
\ref{thm:quotient} a \emph{quotient quasiparabolic $W$-set of $X$}. In Section
\ref{sec:e88}, we will prove the set $\Omega_{U(E_8,8)}$ from the introduction
forms a quasiparabolic $W(E_7)$-set by realizing it as a quotient quasiparabolic
$W(E_7)$-set of another suitable quasiparabolic $W(E_7)$-set (Theorem
\ref{thm:e88}).

\section{Further examples of QP-triples}\label{sec:further}
We discuss two families of examples of QP-triples in this section. The first
family is associated to the root systems of rank $n$ where $n$-roots exist, and
we show in Section \ref{sec:nroots} that for these root systems we can form a
QP-triple by taking $U$ to be the entire set of positive roots. In Section
\ref{sec:p_triples}, we show that for each root system of simply-laced type
$\Gamma$, it is often possible to construct a QP-triple $T(\Gamma,p)$ by taking
$U$ to be the set of all positive roots in which a particular simple root
$\al_p$ appears with odd coefficient (Definition \ref{def:unp}). This
construction furnishes a large number of QP-triples that form our second family
of examples, and we will explain their rich connections to the combinatorics of
rook configurations and labelled Fano planes. Section
\ref{sec:e88} investigates the QP-triple $T(E_8,8)$ as a case study for the
QP-triples of the form $T(\Gamma,p)$. We note
that the QP-triples from both families of examples also give rise to
certain binary matroids (Remark \ref{rmk:matroid}), and that the QP-triples
$T(\Gamma,p)$ from the second family of examples have close relationships with
symmetric spaces and symmetric pairs (Remark
\ref{rmk:symmspace}).

\subsection{Orthogonal bases}\label{sec:nroots}
Recall from Lemma \ref{lem:nroots} that a Weyl group $W$ of rank $n$ and type
$ADE$ has a set of $n$ orthogonal roots if and only if $W$ has type $E_7$,
$E_8$, or $D_n$ for $n$ even. We will show that $(W, S, \posroots)$ is a
QP-triple in these cases. These examples are particularly useful because many
other examples of QP-triples can be constructed from them by restriction and/or
taking quotients. The
next result develops some key properties of the generalized Rothe diagrams of
$n$-roots in these cases. Recall from Remark \ref{rmk:int_res} (i) that for each
$n$-root $R$ of $\Phi$, we have $\res(R)=\res_{\Phi_+}(R)$ and
$\rho(R)=\rho_{\Phi_+}(R)$.

\begin{prop}\label{prop:partition} 
Let $\Phi$ be a root system of type $E_7$, $E_8$, or $D_n$ for $n$ even. If $R$
is an orthogonal set of $n$ positive roots in $\Phi$, then 
the generalized Rothe diagram $\res(R)$ can be expressed as the
disjoint union $$ \res(R) = \dot\bigcup_{Q \in D_4(R)} \res(Q) .$$
\end{prop}

\begin{proof}
Remark \ref{rmk:negate} (i) and Proposition \ref{prop:qtfae} (ii) imply that
every $\gamma \in \res(R)$ lies in the support of a unique $Q \in D_4(R)$. It
follows that $\gamma \in \res(Q)$ in this case, which proves $\res(R) \se
\dot\bigcup_{Q \in D_4(R)} \res(Q)$. Conversely, any element $\gamma \in
\res(Q)$ is fixed by all reflections $s_\be$ such that $\be \in R \backslash Q$,
which implies that $\gamma \in \res(R)$ and establishes the reverse containment.
The disjointness assertion follows from Proposition \ref{prop:qtfae} (iii),
which completes the proof.
\end{proof}

\begin{cor}
\label{cor:c2n}
Let $\Phi$ be a root system of type $E_7, E_8$, or $D_n$ for $n$ even. If $R$ is
a maximal orthogonal set of positive roots in $\Phi$, then $\rho(R)=C(R)+2N(R)$,
where $C(R)$ and $N(R)$ are the numbers of crossings and nestings in $R$,
respectively. 
\end{cor}

\begin{proof}
This is immediate from Proposition \ref{prop:partition} and Remark
\ref{rmk:int_res} (ii). \end{proof}

\begin{theorem}\label{thm:nrtriples} 
Let $W$ be a Weyl group of type $E_7$, $E_8$, or $D_{n}$ for $n$ even,
with root system $\allroots$ and generating set $S$. The triple $(W, S,
\posroots)$ is a transitive QP-triple. \end{theorem}

\begin{proof} 
By Lemma \ref{lem:nroots} and the paragraph preceding it, the set
$\Omega_\posroots$ consists of the positive $n$-roots of $\Phi$, and the
standard action of $W_S=W$ on $\Omega_\posroots$ is transitive. It is known
\cite[Theorem 4.5]{gx5} that $\Omega_\posroots$ is a quasiparabolic set for $W$
under the standard action and level function $C+2N$, where $C$ and $N$ count the
numbers of crossings and nestings in each $n$-root. Since we have
$C(R)+2N(R)=\rho(R)=\rho_{\Phi_+}(R)$ for all $R\in \Omega_\posroots$ by
Corollary \ref{cor:c2n}, it follows that $(W,S,\Phi_+)$ is a transitive
QP-triple, which completes the proof. \end{proof}

\begin{exa}
\label{exa:Dnr}
In type $D_{2k}$, the quasiparabolic set $\Omega_{\Phi_+}$ for
the Weyl group $W=W(D_{2k})$ can
be modelled by the perfect matchings of the set $[2k]=\{1,2,\dots, 2k\}$ as in
the case of the quasiparabolic set $\Omega_U$ for
the parabolic subgroup $W_I\cong S_{2k}$ from Theorem \ref{thm:hafpfaf} (where
$I=S(D_{2k}\setminus\{s_{2k}\}$), with every $2k$-root in $\Omega_{\Phi_+}$
taking the form $R'(M)= \{\ep_{i_1}\pm\ep_{j_1}, \ldots, \ep_{i_k}\pm\ep_{j_k}\}$ for a perfect matching 
$M=\{\{i_1,j_1\},\dots, \{i_k,j_k\}\}$.
Indeed, $\Omega_{\Phi_+}$ is essentially the same quasiparabolic set as
$\Omega_U$, in the sense
that there is a natural bijection
from $\Omega_U$ to $\Omega_{\Phi_+}$ that is $W_I$-equivariant and preserves
level, namely, the map \[\phi: \Omega_U\ra \Omega_{\Phi_+}, \;
  R(M):=\{\ep_{i_1}+\ep_{j_1}, \ldots, \ep_{i_k}+\ep_{j_k}\}  \mapsto
R'(M).\] 
Both $R(M)$ and $R'(M)$ have level $c+2n$, where $c$ and $n$ are the numbers of
crossings and nestings in $M$: we have $\rho_U(R(M))=\ell(\sigma_M)=c+2n$ by
Theorem \ref{thm:hafpfaf} (v) and the last paragraph in the proof of Lemma
\ref{lem:sigmatau}, and we have $\rho(R'(M))=c+2n$ by Corollary \ref{cor:c2n}
and Remark \ref{rmk:acndef}. 

Furthermore, if we consider the fixed-point-free involutions
$\tau_M=(i_1,j_1)\cdots (i_k,j_k)\in S_{2k}$ naturally associated to the perfect
matchings as in Section \ref{sec:motexa}, then the quasiparabolic order $\le$
on $\Omega_{\Phi_+}$ coincides with the restriction of the strong Bruhat order
on $S_{2k}$ to the fixed-point-free involutions under the bijective
correspondence $\tau_M\leftrightarrow R'(M)$, for the following reasons. We have
$\rho(R'(M))=\rho_U(R(M))=(\ell(\tau_M)-k)/2$ by Theorem \ref{thm:hafpfaf} (v),
so if $\tau_M\le r(\tau_M)$ for any reflection $r=(i,j)\in S_{2k}$ then we also
have $R'(M)\le R'(r(M))=r(R'(M))$. Conversely, for any root
$\al=\ep_i\pm \ep_j\in \Phi$, the reflection $r=s_\al$ acts in the same way on
$\Omega_{\Phi_+}$ as the transposition $(i,j)=s_{\ep_i-\ep_j}\in S_{2k}$
(because roots always appear in couples $\ep_a\pm\ep_b$ in the elements of
$\Omega_{\Phi_+}$), so if $R'(M)\le r(R'(M))=(i,j)(R'(M))=R'((i,j)(M))$ then we
have $\tau_M\le (i,j)\tau_M$.  
\end{exa}

\begin{rmk}\label{rmk:nrtriples} 
\begin{itemize}
 \item[(i)] The Poincar\'e polynomial $PS_{\Phi_+}(q)$ factors into quantum
   integers, with the multiset of degrees being $\{3, 5, \ldots, 2k-1\}$ in type
   $D_{2k}$, $\{3, 5, 9\}$ in type $E_7$, and $\{3, 5, 9, 15\}$ in type $E_8$.
   In particular, the lowest-degree term of $PS_{\Phi_+}(q)$ is the constant $1$
   in all the types. The multisets of degrees were obtained computationally in
   \cite[Proposition 7.1 (i)]{gx5}, but we will explain how the degrees can be
   obtained combinatorially for type $D_{2k}$ in Example \ref{exa:aiii}, and a
   computer-free proof of the degrees in types $E_7$ and $E_8$ can now be found in
   \cite{gx7}. We will also explain an algorithm for computing the degrees by
   using a matroid associated with $\Omega_{\Phi_+}$ in Section
   \ref{sec:conclusion}.

 \item[(ii)] Table \ref{table:nrs_numbers} summarizes the key numerical
   quantities associated with the examples of Theorem \ref{thm:nrtriples}, where
   $\kappa=n$ is the common cardinality of the sets in $\Omega_U$. The last
   column of the table concerns the matroid we can naturally associate to
   $\Omega_{\Phi_+}$, whose construction we will explain in Remark
   \ref{rmk:matroid}. Note that the cardinality of the set $\Omega_U$ is simply
   the product of the degrees by Definition \ref{def:residues} (iv), so
   $\abs{\Omega_U}$ equals $(2k-1)!!, 3\cdot5 \cdot 9=135$, and $3\cdot 5\cdot
   9\cdot 15=2025$ in types $D_{2k}, E_7$, and $E_8$, respectively.

   \begin{table}[!ht]
\centering
\begin{center}
\begin{tabular}{ |l|c|c|l|l| }
 \hline
 $\Gamma$ &  $\kappa$& $|U|$  &  degrees & matroid\\
 \hline
 $D_{2k}$ & $2k$ & $2k(2k-1)$  &   $\{3, 5, 7, 
 \ldots, 2k-1\}$ &  $U^{(2)}_{k-1,k}$\\
 $E_7$ & $7$ & $63$ &  $\{3, 5, 9\}$ & $PG(2,2)$\\
$E_8$ & $8$ & $120$ &$\{3, 5, 9, 15\}$& $AG(3,2)$\\
 \hline
\end{tabular}
\caption{QP-triples of the form $(W,S,\posroots)$}
\label{table:nrs_numbers}
\end{center}
\end{table}

\item[(iii)]{The unique $W$-minimal element in type $E_8$ is the set $$ R=
    \{\al_2, \al_3, \al_5, \al_7, \theta_4, \theta_6, \theta_7, \theta_8\},$$
    where $\theta_4$, $\theta_6$, $\theta_7$, and $\theta_8$ are the highest
    roots in the standard parabolic subsystems of types $D_4$, $D_6$, $E_7$, and
    $E_8$ of $\Phi$, respectively. This follows from \cite[\S6.3]{gx5}, and can
    also be checked directly using the definition of $\res(R)$. Since the
    lowest-degree term of $PS_{\Phi_+}(q)$ is the constant 1 by (i), we have
  $\rho(R)=0$.}
\end{itemize}
\end{rmk}

Recall from Proposition \ref{prop:qtfae} that if $R$ is an $n$-root of type,
then for each root $\al\in \Phi_+\setminus R$ the reflection $s_\al$ changes
precisely four elements in $R$ that form a coplanar quadruple $Q$. The next
lemma shows that the effects of $s_\al$ on $R$ and $Q$ are consistent with
respect to the quasiparabolic orders on $\Omega_{\Phi_+}$ and on
$\Omega_{\Psi}$, where $\Psi$ is the $D_4$-subsystem spanned by $Q$.  This
consistency will be very useful in the proofs of Section \ref{sec:e88}. 

\begin{lemma}\label{lem:poseven}
Let $\Phi$ be a root system of type $E_7$, $E_8$, or $D_{n}$ for $n$ even. Let
$R$ be a positive $n$-root of $\Phi$, let $\al \in \posroots \backslash R$, and let
$Q=\{\be\in R: B(\al,\be)\neq 0\}\se R$ be the quadruple of elements not fixed
by $s_\al$ in $R$. Let $\Psi=\Span(Q)\cap \Phi$ be the $D_4$-subsystem
spanned by $Q$, and set $R_i=(R\setminus Q)\cup Q_i$, where $Q_0, Q_1$ and $Q_2$
denote the alignment, crossing, and nesting in $\Psi$, respectively. 
\begin{enumerate}
\item 
We have $\rho(s_\al Q) \ne \rho(Q)$.
\item 
We either have $\rho(s_\al Q)<\rho(Q)$ and 
$ \rho(s_\al R) < \rho(R)$, or we have 
$\rho(s_\al Q)>\rho(Q)$ and 
$ \rho(s_\al R) > \rho(R)$.
\item We have $R_0< R_1< R_2$.  
\end{enumerate}
\end{lemma}

\begin{proof}
The quadruples $Q$ and $s_\al Q$ are distinct coplanar quadruples in the 
$D_4$-subsystem $\Psi$, so they have different levels by Example
\ref{exa:residues}. This proves (i).

Part (ii) follows from Remark 3.12 and Proposition 4.7 of \cite{gx5}. By that
remark and Remark \ref{rmk:int_res} (ii), if $\rho(s_\al Q)>\rho(Q)$ then we
have (a) $Q$ is an alignment and $s_\al Q$ is a crossing, or (b) $Q$ is a
crossing and $s_\al Q$ is a nesting, or (c) $Q$ is an alignment and $s_\al Q$ is
a nesting. In these cases, the number $d:=\rho(s_\al R) - \rho(R)$ is positive
by parts (ii), (iii), and (iv) of \cite[Proposition 4.7]{gx5}, respectively, so
we have $\rho(s_\al R)>\rho(R)$. This proves that if $\rho(s_\al Q)>\rho(Q)$
then $\rho(s_\al R)>\rho(R)$.

If $\rho(s_\al Q)\not>\rho(Q)$, then we have $\rho(s_\al
Q)<\rho(Q)=\rho(s_\al(s_\al Q))$ by (i).
The set of elements in the $n$-root $R':=s_\al R$ that are not fixed by $s_\al$
is precisely $Q':=s_\al Q$, because for any $\be\in R$ we have
$B(\al,s_\al(\be))=B(s_\al(\al),\be)=B(-\al,\be)=-B(\al,\be)$. Applying the
arguments of the previous paragraph to $R'$ and $Q'$ now shows that $\rho(s_\al
R)<\rho(R)$, which completes the proof of (ii).

By Example \ref{exa:residues}, we have $\rho(Q_i)=i$ for each $i\in \{0,1,2\}$,
and we have $Q_1=s_2(Q_0)$ and $Q_2=s_1(Q_1)$ for the simple reflections $s_1$
and $s_2$ associated to the simple roots $\al_1$ and $\al_2$ of $\Psi$. These
reflections fix the roots of $R\setminus Q$ because $\al_1,\al_2\in
\Psi\se\Span(Q)$, so we have $R_1=s_2(R_0)$ and $R_2=s_1(R_1)$, and the
conclusion of (iii) now follows from (ii) and the definition of the
quasiparabolic order $\le$.
\end{proof}

\subsection{An effective construction}\label{sec:p_triples}
The following definition produces a highly effective method for constructing
QP-triples from simply-laced root systems. As our examples will show, many of
the QP-triples constructed using the definition recover known objects, in some
cases providing extra insight into their structure. The sets $U(\Gamma,p)$ in
the definition correspond to symmetric spaces in a certain precise way that we
will explain in Remark \ref{rmk:symmspace}. 

\begin{defn}\label{def:unp}
Let $W$ be a Weyl group of type $ADE$ with Dynkin diagram $\Gamma$, generating
set $S$, and root system $\allroots$. Let $\al_p$ be a simple root of
$\allroots$. We define $U(\Gamma, p) \subseteq \posroots$ to be the set of all
roots in which $\al_p$ appears with odd positive coefficient. We define
$\Omega(\Gamma,p):=\Omega_{U(\Gamma,p)}$ and define $T(\Gamma, p)$ to be the
triple $$T(\Gamma,p)=(W, S \backslash \{s_p\}, U(\Gamma, p)) .$$ \end{defn}

Note that when we expand a root in $\Phi$ into a linear combination of simple
roots, the coefficient of each simple root lies in the set $\{\pm 1,0\}$ if
$\Phi$ is of type $A$ and in the set $\{\pm 2,\pm 1,0\}$ if $\Phi$ has type
$D_n$. Thus, in these two types, the condition in Definition \ref{def:unp} that
$\al_p$ appears with odd positive coefficient is equivalent to the condition
that $\al_p$ appears with coefficient 1.

\begin{table}[!ht]
\centering
\begin{center}
\begin{tabular}{|l|c|c|c|l|c|}
 \hline 
 $\Gamma$ & $p$ &  $\kappa$ & $\abs{U}$ &  degrees & matroid\\
 \hline
 $A_{2k-1}$ &  $k$ & $k$ & $k^2$ & $[p]_{p-1}$ &
 $U_{p-1,p}$  \\
 $A_{n}$ & $[n]\setminus\{(n+1)/2\}$  & $m:=\min(p,n+1-p)$ & $m(n+1-m)$ & $[n+1-m]_{m}$ & $U_{m,m}$  \\
 $D_{2k}$ &  $k$ & $2k$ & $2k^2$ & $[p]_{p-1}$ & $U^{(2)}_{p-1,p}$  \\
 $D_{n}$ & $ 1\le p<n/2$  & $2p$ & $2p(n-p)$ & $[n-p]_p$ &
 $U^{(2)}_{p,p}$\\
 $D_{2k}$ & $2k-1, 2k$ & $k $  & $k(2k-1)$ & $\{3,5,\ldots, 2k-1\}$& $U_{k-1,k}$\\ 
 $D_{2k+1}$ & $2k, 2k+1$ & $k $  & $k(2k+1)$ & $\{3,5,\ldots, 2k+1\}$ & $U_{k,k}$\\ 
$E_6$ & $2$ & $4$ & $20$ & $\{2,3,5\}$ & $AG(2,2)$\\
$E_6$ & $1,6$ & $2$ & $16$ & $\{5,8\}$& $AG(1,2)$\\
$E_7$ & $2,5$ & $7$ & $35$ & $\{2,3,5\}$ & $PG(2,2)$\\
$E_7$ & $1$ & $4$ & $32$ & $\{3,5,8\}$ & $AG(2,2)$\\
$E_7$ & $7$ & $3$ & $27$ & $\{5,9\}$ & $PG(1,2)$\\
$E_8$ & $1,2,5,6$ & $8$ & $64$ & $\{2,3,5,8\}$ & $AG(3,2)$\\
$E_8$ & $8$ & $4$ & $56$ & $\{5,9,14\}$ & $AG(2,2)$\\
\hline
\end{tabular}
\caption{QP-triples $T(\Gamma, p)$ of type $ADE$}
\label{table:tnp_numbers}
\end{center}
\end{table}

Table \ref{table:tnp_numbers} lists all QP-triples of the form $T(\Gamma,p)$
where $\Gamma$ is of type $ADE$. In the table, $U$ stands for $U(\Gamma,p)$,
$\kappa$ stands for the maximum cardinality of orthogonal subsets of $U$, $k$
stands for an integer larger than $1$, $n$ denotes the rank of $\Phi$, and
$\lfloor \cdot \rfloor$ denotes the floor function. The Poincar\'e polynomial
factors as a multiplicity-free product of quantum integers in all cases, so each
example has a well-defined, multiplicity-free set of degrees; we express this
set in the notation \[ [a]_b:=\{a,a-1, a-2, ..., a-b+1\} \] in some rows of the
table. (Recall that the cardinality of the quasiparabolic set $\Omega_U$ can be
deduced from the degrees, as the product of the degrees, by Definition
\ref{def:residues} (iv).) The last column of the table describes the matroid we
can naturally associate to $\Omega_{U}$, whose construction we will
elaborate on in Remark \ref{rmk:matroid}. 
We have verified the correctness
of the table, as well as the fact that it exhausts all QP-triples in type $ADE$,
by computation. We note that by our computation, in the cases where $T(\Gamma,
p)$ is not a QP-triple, it always happens that $T(E_n, p)$ is not even a scaled
$W_I$-set. 

In the rest of this subsection, we will explain how to interpret all the
examples in types $A$ and $D$ from Table \ref{table:tnp_numbers} using the
combinatorics of rook configurations, as well as how some of the examples in
type $E$ can be understood in terms of labelled Fano planes. We will give a
conceptual proof of the fact that $T(E_8,8)$ is a QP-triple, as a case study in
type $E$, in the next subsection. 

\begin{exa}[Type $A$]\label{exa:aiii}
The triple $T(A_n, p)$ is a QP-triple for all integers $n\ge 1$ and $1\le p\le
n$. For such a triple, the set $U=U(A_n,p)$ consists of the roots of the form
$\ep_i-\ep_j$ where $1\le i\le p<j\le n+1$, so we may identify $U$ with the
positions on a $p\times (n+1-p)$ board with rows labelled by 1 to $p$ from top to
bottom and columns labelled by $p+1$ to $n+1-p$ from left to right. Each element
of $R\in \Omega_{U(A_n, p)}$ consists of $m:=\min(p, n+1-p)$ orthogonal roots,
and if we view $R$ as an involution in the natural way, as the product $w_R\in
S_{n+1}$ of the 2-cycles $(i,j)$ for which $\ep_i-\ep_j\in R$, then a root
$\ep_a-\ep_b\in U(A_n,p)$ dominates $R$ if and only if we have both $w_R(a)<b$
and $w_R(b)>a$. Equivalently, a position on the board lies in $\res_U(R)$ if and
only if it is not attacked by a rook to its right or above it in the rook
configuration corresponding to $R$, as shown in Figure \ref{fig:rothe_A}, where
the elements of $\res_U(R)$ are marked by stars as in Figure \ref{fig:og_rothe}.
Thus, the generalized Rothe diagrams associated to $T(A_n,p)$ highly resemble
the Rothe diagrams mentioned in Remark \ref{rmk:biset} (i). 

\begin{figure}[h!]
\centering
\begin{tikzpicture}[scale=0.7]
    \draw (0,0) grid (5,3);

    \foreach \x in {4,...,8}
        \draw (\x-3.5, 3.3) node {$\x$};
    \foreach \y in {1,...,3}
        \draw (-0.3, 3.5-\y) node {$\y$};
      
    \node at (3.5, 2.5) {$\bigstar$};
    
    \node at (4.5, 2.5) {$\bigstar$};
        \node at (4.5, 1.5) {$\bigstar$};
    \node at (4.5, 0.5) {$\bigstar$};

    \filldraw[black] (3.5, 1.5) circle (4pt); 
    \filldraw[black] (1.5, 0.5) circle (4pt); 
    \filldraw[black] (2.5, 2.5) circle (4pt); 
\end{tikzpicture}

\caption{The generalized Rothe diagram of $\{\ep_1-\ep_6,
\ep_2-\ep_7, \ep_3-\ep_5\}\in \Omega_{U(A_7, 3)}$}
\label{fig:rothe_A}
\end{figure}

Using the rook configuration model, one can show using an elementary
combinatorial argument that the Poincar\'e polynomial of $U(A_n,p)$ is
$\prod_{d\in [n+1-m]_m}[d]_q$, as asserted in Table \ref{table:tnp_numbers}. The
row--column symmetry of the model implies that there is a natural
level-preserving bijection between the quasiparabolic sets $\Omega(A_n,p)$ and
$\Omega(A_n,n+1-p)$; in particular, the Poincar\'e polynomials of $U(A_n,p)$ and
$U(A_n,n+1-p)$ agree. \end{exa}

\begin{rmk}
\label{rmk:garsia_remmel}
The rook configuration model described in Example \ref{exa:aiii} is inspired by
a similar model used by Garsia and Remmel in \cite{garsia86} to study a
$q$-analogue of the Stirling numbers of the second kind. Garsia and Remmel place
rooks on a general Ferrers board, and our model corresponds to the special case
where the board is an $a\times b$ rectangle containing $\min(a,b)$ non-attacking
rooks, with the set $\res_U(R)$ in our construction corresponding to the
positions marked with a circle on the board; see \cite[Figure (I.7)]{garsia86}.
\end{rmk}

\begin{exa}[Type $D$] \label{exa:di}
For all integers $n\ge 4$ and  $1\le p\le n$, the triple $T(D_n, p)$ is a
QP-triple whenever we have  $1\le p\le \lfloor n/2\rfloor$, but not when
$\lfloor n/2\rfloor <p< n-2$. However, if we relax the setting of Definition
\ref{def:unp}, then the triple\[
T'(\Gamma,p)=(W(D_{n}),S(D_{n})\setminus\{s_{p}\}, U'_p:=U(D_n,p)\cap U(D_n,
n)),\] forms a QP-triple for all $1\le p\le n-2$. Note that the QP-triple $(W,
I, U)$ from Theorem \ref{thm:perdet} is a special case of the triple
$T'(\Gamma,p)$ from (v): it is $T'(D_{2k},k)$.

We note that for all the QP-triples  mentioned in the last paragraph, the
maximum-cardinality orthogonal sets $R$ and their generalized Rothe diagrams can
be interpreted in terms of rook configurations as in Example \ref{exa:aiii}, but
on a $p\times (n-p)$ board and using the rules of the Rothe diagrams of Theorem
\ref{thm:perdet}, in which the attacks come from above and to the left. As with
Example \ref{exa:aiii}, the rook configuration model can be used to give an
elementary combinatorial proof of the factorization of the Poincaré polynomial,
but this comes with the caveat that a position on the board may now correspond
to either one of two roots $\ep_i\pm \ep_j$. The fact that $T(D_n, p)$ fails to
be a QP-triple when $\lfloor n/2\rfloor <p< n-2$ has to do with this caveat and
the fact that in these cases $R$ does not contribute a rook to every row. By
contrast, intersecting $U(D_n,p)$ with $U(D_n,n)$ to produce $U'_p$ eliminates
this complication for $T'(D_n,p)$: all elements of $U(D_n,n)$ are of the form
$\ep_i+\ep_j$, so the positions on the board are now in a natural one-to-one
correspondence with the elements of $U'_p$, which is why $T'(D_n,p)$ does form a
QP-triple for all $1\le p\le n-2$ in a similar way to Example \ref{exa:aiii}.
\end{exa}

\begin{exa}[Type $D$]\label{exa:diii}
For any integer $k\ge 2$, the triples $T(D_{2k},2k)$, $T(D_{2k}, 2k-1)$,
$T(D_{2k-1},2k-1)$, and $T(D_{2k-1},2k-2)$ are all QP-triples. Note that they
are different from the QP-triples in Example \ref{exa:di}. 

Of the above four triples, the first triple $T(D_{2k},2k)$ is precisely the
triple $(W, I, U)$ appearing in Theorem \ref{thm:hafpfaf}. The other three
triples give rise to essentially the same quasiparabolic set as $T(D_{2k}, 2k)$,
in the sense that their associated quasiparabolic sets $\Omega_{U(\Gamma,p)}$
all admit canonical level-preserving bijections to the quasiparabolic set
$\Omega(D_{2k}, 2k)$. This bijection is induced by the symmetry of the Dynkin
diagram for $\Omega(D_{2k},2k-1)$. For the case of
$\Omega=\Omega(D_{2k-1},2k-1)$, each element of $\Omega$ corresponds to a
(non-perfect) matching of $\{1, 2,\dots, 2k-1\}$ into $(k-1)$ pairs and one
unmatched point, and the canonical bijection from $\Omega$ to $\Omega(D_{2k},
2k)$ sends such an element to the $2k$-root in $\Omega(D_{2k},2k)$ corresponding
to the unique perfect matching containing those $(k-1)$ pairs as blocks.
\end{exa}

%
\begin{exa} [Type $E$]
\label{exa:fano7}
The structures of the QP-triple $T(E_7, 2)$ and the associated quasiparabolic
set $\Omega(E_7,2)$ are known by \cite[Section 6.2]{gx5}: the parabolic subgroup
$W_I$ has type $A_6$, the set $\Omega(E_7,2)$ coincides with the set of positive
$7$-roots of $\Phi(E_7)$ containing no alignments, and $W_I\cong S_7$ acts
transitively on $\Omega(E_7, 2)$ in the standard action. To each root $\al\in
U(E_7,2)$ we may naturally associate a triple $t_\al$ of integers from the set
$[7]:=\{1,2,..., 7\}$, and $\Omega(E_7, 2)$ can be naturally identified with the
$30$ inequivalent labellings of the Fano plane using the set $[7]$. It can be
shown that a root $\al\in U(E_7,2)$ dominates a labelling $L$ if and only if
the elements of $t_\al$ do not form a line in $L$ and each element in $t_\al$ is
larger than the unique element collinear with the other two in $L$. For example,
the unique $W_I$-minimal labelling $$ \{136, 145, 127, 235, 246, 347, 567\} $$
is the only labelling dominated by no root, and the unique $W_I$-maximal
labelling $$ \{123, 145, 246, 257, 347, 356, 167\} $$ is dominated by seven
roots: $267$, $357$, $367$, $456$, $457$, $467$, and $567$. The generalized
Rothe diagram of each labelling can be interpreted as the set of non-positions
in the context of games arising from the Steiner system $S(2, 3, 7)$; see
\cite{irie21}.
\end{exa}

\begin{exa} [Type $E$]
\label{exa:fano8}
For the QP-triple $T(E_8, 2)$, the results of \cite[Section 6.3]{gx5} show that
the corresponding subgroup $W_I$ has type $A_7$, the set $\Omega(E_8,2)$
coincides with the set of positive $8$-roots of $\Phi(E_8)$, and $W_I\cong S_8$
acts transitively on $\Omega(E_8,2)$ in the standard action.  The set
$\Omega(E_8, 2)$ is in natural bijection with the set of $240$ inequivalent
labellings of the Fano plane by the set $\{1, 2, \ldots, 8\}$, as well as with
the 240 packings of the projective space $PG(3, 2)$.  The generalized Rothe
diagram of each $8$-root in $\Omega(E_8,2)$ has a characterization similar to
the one in Example \ref{exa:fano7}, although the details are more complicated.
This particularly intricate example is the subject of the paper \cite{the240}.
\end{exa}

\subsection{The QP-triple \texorpdfstring{$T(E_8, 8)$}{T(E8,8)}}
\label{sec:e88}
Let $\Phi$ be the root system of type $E_8$, let $W$ be the Weyl group of $\Phi$
with generating set $S$, set $I=S\setminus\{s_8\}$, and write $U=U(E_8,8)$. The
goal of this subsection is to give a non-computational proof of the fact that
$T(E_8,8)=(W,I,U)$ is a QP-triple, i.e., that $\Omega_{U}$ is a quasiparabolic
set for $W_I\cong W(E_7)$ (Theorem \ref{thm:e88}).  

Let $U_i=\Phi_{8,i}\se \Phi$ be the set of positive roots of $8$-height $i$ for
each $i\in \{0,1,2\}$. Recall from Section \ref{subsec:roots} that the highest
root $\theta$ of $\Phi$ is the only element of $U_2$, so $U=U_1$ by the
definition of $U(E_8,8)$.

Our strategy for proving $T(E_8,8)$ is a QP-triple is as follows. Recall from
Remark \ref{rmk:complement} that roots in $\Phi$ orthogonal to the highest root
$\theta$ forms a subsystem $\Phi_\theta$ of type $E_7$. Indeed, this subsystem
is simply the $E_7$-subsystem spanned by the simple roots $\al_1,
\cdots,\al_7$, because these simple roots are all orthogonal to $\theta$. The
$8$-roots in $\Omega_{\Phi_+}$ containing $\theta$ are in bijection with the 135
positive $7$-roots $R'\se \Phi_\theta$ via the map $R'\mapsto R'\cup
\{\theta\}$. The other 8-roots of $\Omega_{\Phi_+}$ form a subset 
\[
  \Omega:=\{R\in \Omega_{\Phi_+}: \theta\notin R\}
\]
of $\Omega_{\Phi_+}$ with $2025-135=1890$ elements. The set
$\Omega$ inherits the quasiparabolic order of $\Omega_{\Phi_+}$, and it is 
invariant under the $W_I$ action because $s_i(\theta)=\theta$ for any $s_i\in I$.
We will show that the equivalence relation 
$\approx$ on $\Omega$ defined by 
\[
  R\approx R' \quad \text{if} \quad R\cap U_1=R'\cap U_1
\]
is a $W_I$-invariant poset congruence on $\Omega$, and then we will use Theorem
\ref{thm:quotient} to identify $\Omega_U$ with the resulting quotient
quasiparabolic set.

\begin{lemma}\label{lem:e8roots}
Let $R\in \Omega$, and let
$\resi{i}(R)=\res(R)\cap U_i$ for each $i\in \{0,1,2\}$.
\begin{itemize}
\item[{\rm (i)}]{We have $R = Q_1 \ \dot\cup \ Q_0$, where $Q_i$ is a feature
  consisting of roots in $U_i$ for each $i$, and $Q_1$ is the support of $\theta$
  with respect to the orthogonal basis $R$.}

\item[{\rm (ii)}] Let $\Psi=\Span(Q_1)\cap \Phi$ be the $D_4$-subsystem of $\Phi$
  spanned by $Q_1$, and let 
\[
  \Psi^\perp=\{\al\in \Phi:B(\al,\be)=0\;\text{for all }\be\in \Psi\}
\]
be the unique $D_4$-subsystem of $\Phi$ consisting of roots orthogonal to $\Psi$, as in
Lemma \ref{lem:overlap} (ii). The $\approx$-equivalence class $[R]$ of $R$
consists of the elements $R_i=Q_1\ \dot\cup \ F_i$ for $i\in \{0,1,2\}$, where
$F_0, F_1,$ and $F_2$ are the alignment, crossing, and nesting in $\Psi^\perp$,
respectively. 

\item[{\rm (iii)}] We 
have \[\resi{2}(R)=\{\theta\},\quad \resi{1}(R) =\res_{U}(Q_1),\
  \text{and}\quad \resi{0}(R)=\res(Q_0).\] 
\item[\rm (iv)]
  We have $\rho(R_i)= \rho_U(Q_1)+i+1$ 
  for each $i$. In particular, the minimum element of $[R]$ is $R^-=R_0$, and we
  have 
  \[
  \rho(R^-)=\rho(R_0)=\rho_U(Q_1)+1.\]
\end{itemize}
\end{lemma}

\begin{proof}
Since $\theta \not\in R$, every element of $R$ has $8$-height $0$ or $1$. By
Proposition \ref{prop:qtfae}, the support of $\theta$ with respect to $R$ is a
feature $Q_1 := \{\be_1, \be_2, \be_3, \be_4\}$ such that $\pm \be_1 \pm \be_2
\pm \be_3 \pm \be_4 = 2\theta$, and $\theta$ is orthogonal to the other four
elements of the set $Q_0:=R\setminus Q_1$. Since $\theta$ has $8$-height 2, it
must be the case that $\be_1+\be_2+\be_3+\be_4=2\theta$, so $Q_1$ is a feature
and we have $Q_1\se U_1$. Lemma \ref{lem:overlap} (ii) implies that $Q_0$ is a
feature because $Q_0=R\setminus Q_1$, and we have $Q_0\se U_0$ because
$B(\theta,\gamma)=0$ for all $\gamma\in Q_0$. This proves (i). 

The set $Q_0$ is a feature in $\Psi^\perp$ by (i), and (ii) then follows from
Example \ref{exa:residues}.

The root $\theta$ is the only root of $8$-height $2$ and has strictly greater
$8$-height than any root in $R$. This implies that $\resi{2}(R) = \{\theta \}$.

If $\beta$ has $8$-height $0$ and $\gamma$ has $8$-height $1$, then it is not 
possible for $s_\beta(\gamma)$ to be a negative root. This implies that
$\res_{U_1}(Q_1) \se \resi{1}(R)$. The converse inclusion $\resi{1}(R) \se 
\res_{U_1}(Q_1)$ is immediate from the definitions, which proves that 
$\resi{1}(R) =\res_{U_1}(Q_1)=\res_U(Q_1)$.

If $\beta$ has $8$-height $1$, $\gamma$ has $8$-height $0$, and 
$B(\beta, \gamma) = 0$, then $s_\beta(\gamma) = \gamma > 0$. This implies that
$\res(Q_0) \se \resi{0}(R)$. To prove that $\resi{0}(R) \se \res(Q_0)$, it
suffices to show that $\resi{0}(R) \se \Span(Q_0)$.

Let $\gamma \in \resi{0}(R)$. Since $\gamma$ has $8$-height $0$, it follows that
$\gamma$ is orthogonal to $2\theta = \be_1 + \be_2 + \be_3 + \be_4$. If we have
$B(\gamma, \be_i)=0$ for all $i$, then we have $\gamma \in \Span(Q_0)$ as
required. Otherwise, we must have $B(\gamma, \be_i) = +1$ and $B(\gamma,
\be_j)=-1$ for some $1 \leq i \ne j \leq 4$. Since $\be_i$ has $8$-height $1$
and $\gamma$ has $8$-height $0$, this implies that
$s_{\be_i}(\gamma)=\gamma-\beta_i$ has $8$-height $-1$ and is a negative root.
This contradicts the hypothesis that $\gamma \in \res(R)$, proving (iii).

Part (iv) follows from (i)--(iii) and the fact that $\rho(F_i)=i$ for all $i\in
\{0,1,2\}$ by Example \ref{exa:residues}.
\end{proof}

\begin{cor}
\label{cor:complete}
We have $\Omega_{U}=\{R\cap U_1: R\in \Omega\}=\{Q\se
U_1: Q \text{ is a feature}\}$. In particular, the maximum size of an orthogonal
subset of $U$ is 4.
\end{cor}

\begin{proof}
We have $B(\gamma,\theta)=1\neq 0$ for all $\gamma\in U_1$, so it follows from
Lemma \ref{lem:nroots} that every element $Q\in \Omega_{U}=\Omega_{U_1}$ can be
extended to an $8$-root $R\in \Omega$. The asserted set equalities now follow
from Lemma \ref{lem:e8roots} (i). Since each feature consists of 4 roots, it
follows that the maximum size of an orthogonal subset of $U$ is 4. \end{proof}

By Corollary \ref{cor:complete}, the elements of $\Omega_U$ are coplanar
quadruples $Q$ obtained by taking the components of the 8-roots $R\in \Omega$
with 8-height 1. Recall from Definition \ref{def:D46} that for each $R\in
\Omega$, the notation $D_6(R)$ stands for the 6-element subsets $H$ that span a
$D_6$-subsystem $\Xi=\Xi_H=\Span(H)\cap\Phi$ in $\Phi$. The next two lemmas
concern chains of the form $Q\se H\se R$ where $H\in D_6(R)$. We will
combine the lemmas with Lemma \ref{lem:poseven}, which allows us to relate the effects
of reflections in $W$ on $Q, H$, and $R$, to show that $\approx$ is
a poset congruence on $\Omega$. 
Recall also that by Example \ref{exa:Dnr} and Remark \ref{rmk:nrtriples} (ii), if we
realize $\Xi$ in the usual coordinates, then the set $\Omega_{\Xi_+}$ consists
of 15 positive $6$-roots of the form $R'(M)=\{\ep_{i_p}\pm \ep_{j_p}: 1\le p \le
3\}$. These $6$-roots are in natural bijection with the set $\mathcal{M}$ of
the perfect matchings $M=\{\{i_p,j_p\}: 1\le p\le 3\}$ of the set $[6]$ and with
the fixed-point-free involutions $\tau_M = (i_1,j_1)(i_2,j_2)(i_3,j_3)$ in
$S_6$, and the quasiparabolic order $\le$ on $\Omega_{\Xi_+}$ coincides
with the restriction of the strong Bruhat order of $S_6$ to the fixed-point-free
involutions under the correspondence $R'(M)\leftrightarrow \tau_M$.

\begin{lemma}
\label{lem:fifteen}
Let $\Xi$ be a root system of type $D_6$, and maintain the notation of the previous paragraph. If $\sim$ is the equivalence relation on $\Omega_{\Xi_+}$ defined by \[
    R'(M)\sim R'(M') \quad\text{if}\quad \tau_M(1)=\tau_{M'}(1) \] for any
     $M, M'\in \mathcal{M}$, then $\sim$ is a poset congruence on
     $\Omega_{\Xi_+}$ with respect to the quasiparabolic order $\le$. 
\end{lemma}
\begin{proof}
  Let $M\in \mathcal{M}$ and suppose $\tau_M(1)=j$ for some $j\in [6]$. Let
  $a<b<c<d$ be the four elements of $[6]\setminus\{1,j\}$. There are three
  perfect matchings of $\{a,b,c,d\}$, namely, the alignment
  $F_0=\{\{a,b\},\{c,d\}$, the crossing $F_1=\{\{a,c\},\{b,d\}\}$, and the
    nesting $F_2=\{\{a,d\},\{b,c\}\}$ (Definition \ref{def:acn}). It follows
    that the $\sim$-equivalence class of $R'(M)$ consists of the three 6-roots
    $R'(M_0), R'(M_1), R'(M_2)$, where $M_i=\{\{1,j\}\}\cup F_i$ for each $i$.
    We denote this equivalence class by $X_j$. 

  Let $R'(F_0)=\{\ep_{a}\pm\ep_{b}, \ep_{c}\pm\ep_{d}\}\se R'(M_0)$ be the
  coplanar quadruple corresponding to the alignment $F_0$, and define $R'(F_i)\se
  R'(M_i)$ similarly for $i\in \{1,2\}$. By Remark \ref{rmk:acndef} and Lemma
  \ref{lem:poseven} (iii), we have $R'(M_0)< R'(M_1)< R'(M_2)$ in $\Omega_{\Xi_+}$, so
  $X_j$ is a chain with minimum element $X_j^-=R'(M_0)$ and maximum element
  $X_j^+=R'(M_2)$. 

  If $R'(M_0)\le R'(M') \le R'(M_2)$ for some $M'\in \mathcal{M}$, then
  $\tau_{M_0}\le \tau_{M'} \le \tau_{M_2}$ in the Bruhat order of $S_6$. It is
  known that this implies (by \cite[Theorem 2.1.5]{bjorner05}, for example) that
  $\tau_{M_0}(1)\le \tau_{M'}(1)\le \tau_{M_2}(1)$, which forces
  $\tau_{M'}(1)=j$ and thus $R'(M')\in X_j$. This proves that $X_j$ equals the
  interval $[X_j^-, X_j^+]=[R'(M_0),R'(M_2)]$. 
  More generally, if $M'\in \mathcal{M}$ is any perfect matching such that
  $R'(M)\le R'(M')$, then we have $\tau_M\le \tau_{M'}$ in $S_6$ and thus
  $\tau_M(1)\le \tau_{M'}(1)$, so we have $R'(M')\in X_{j'}$ for some $j'\ge j$. The
  minimum elements of the $\sim$-equivalence classes $X_2^-, X_3^-, \dots, X_6^-$
  correspond to the fixed-point-free involutions \[ (12)(34)(56),\quad 
    (13)(24)(56),\quad  (14)(23)(56),\quad (15)(23)(46),\quad\text{and}\quad
(16)(23)(45), \] 
respectively, and these involutions form an
increasing chain in the Bruhat order (by direct computation, for example).
Similarly, 
the maximum elements of the $\sim$-equivalence classes correspond to the
increasing chain 
$$
(12)(36)(45) < (13)(26)(45) < (14)(26)(35) < (15)(26)(34) < (16)(25)(34)
$$
of fixed-point-free involutions. It follows that $X_j^-\le X_{j'}^-$ and $X_j^+\le
X_{j'}^{+}$, i.e., that $f^-(R'(M))\le f^-(R'(M'))$ and $f^+(R'(M))\le
f^+(R'(M'))$ in the
notation of Definition \ref{def:reading}. This completes the proof that $\sim$ is a poset congruence on $\Omega_{\Xi_+}$.
\end{proof}

\begin{lemma}\label{lem:110000}
Let $R$ be a positive $8$-root of $\Phi$ such that $\theta\notin R$. Let $H \in
D_6(R)$, let $\Xi = \Span(H) \cap \allroots$ be the associated $D_6$-subsystem
of $\Phi$, and suppose that we have $|H \cap U_1| = 2$ and $|H \cap U_0| = 4$.
\begin{itemize}
\item[{\rm (i)}] We have $\theta\notin \Xi$. 
\item[{\rm (ii)}] The highest root $\psi:=\ep_1+\ep_2$ of $\Xi$ has $8$-height
     $1$. 
\item[{\rm (iii)}]{The simple root $\be_1:=\ep_1 - \ep_2$ of $\Xi$ has $8$-height
     $1$; the other simple roots of $\Xi$ have $8$-height $0$.}
\item[{\rm (iv)}]{A root $\gamma \in \Xi_+$ has $8$-height $1$ if $\gamma =
     \ep_1 \pm \ep_i$ for some $i > 1$, and has $8$-height $0$ otherwise.}
\end{itemize}
\end{lemma}

\begin{proof}
The support of $\theta$ in $R$ is the coplanar quadruple $Q_1:=R\cap U_1$ by
Lemma \ref{lem:e8roots} (i), so if $\theta\in \Xi$ then we would have $Q_1\le H$
and $|H \cap U_1|=4$, which contradicts the assumption that $\abs{H\cap U_1}=2$.
This proves (i).

Since $\theta\not\in \Xi$, the highest root $\psi$ of $\Xi$ has $8$-height at
most 1. On the other hand, since $H \cap U_1\neq \emptyset$ by assumption, $H$
contains a root of $8$-height 1 and hence so does $\Xi$. It follows that
$\psi\in U_1$, which proves (ii).   

We now prove the claim that $\be_1\in U_1$. The set $Q_0:=H\cap U_0$ has size 4
by assumption and is disjoint from $Q_1$, so $Q_0=R\cap
U_0$ by 
Lemma \ref{lem:e8roots} (i). Lemma \ref{lem:overlap} (i) implies that $Q_0$ 
contains two couples in $\Xi$ while $H\cap U_1=H\setminus U_0$ consists of one 
couple. We have just proved that $\psi=\ep_1+\ep_2$ lies in $U_1$, so if $\psi\in H$ 
then we have $H\cap U_1=\{\ep_1\pm \ep_2\}$ and thus $\be_1\in U_1$.  If 
$\psi\notin H$, then the perfect matching of the set $[6]$ corresponding to $H$ 
contains two distinct pairs $\{1, j\}$ and $\{2, k\}$. The four roots 
$\ep_1\pm \ep_j, \ep_2\pm \ep_k \in H$ sum to $2\psi$ in this case, so either 
$\{1, j\}$ corresponds to a pair of $8$-height $1$ and $\{2, k\}$ corresponds 
to a pair of $8$-height $0$, or vice versa. On the other hand, the root 
$\beta_1\in \Xi_+\se \Phi_+$ has nonnegative $8$-height and satisfies 
\[2\beta_1 = ((\ep_1 + \ep_j) + (\ep_1 - \ep_j)) - 
((\ep_2 + \ep_k) + (\ep_2 - \ep_k)).\] It follows that 
$\ep_1\pm \ep_j\in U_1, \ep_2\pm \ep_k\in U_0$, and $\be_1\in U_1$, as claimed. 

Since every simple root of $\Xi$ appears in the decomposition of $\psi$ into
simple roots, the fact that $\psi$ has the same $8$-height as $\be_1$ implies that
all simple roots of $\Xi$ other than $\be_1$ lie in $U_0$, proving (iii).

Part (iv) follows from (iii) by expressing $\gamma$ as a linear combination of
simple roots.
\end{proof}

\begin{lemma}\label{lem:1890}
  Let $\Omega=\{R\in \Omega_{\Phi_+}: \theta\notin R\}$ be the set of the $1890$
  $8$-roots in $\Omega_{\Phi_+}$ that do not contain the highest root $\theta$,
  considered as a quasiparabolic $W(E_7)$-set by restricting the action of
  $W=W(E_8)$. The equivalence relation $\approx$ on $\Omega$ for which $R_1
\approx R_2$ if $R_1\cap U_1 = R_2 \cap U_1$ is a $W(E_7)$-invariant poset
congruence with respect to the quasiparbolic order on $\Omega$. 
\end{lemma}

\begin{proof}
The $W(E_7)$-invariance of $\approx$ follows from the fact that the action of
an element in $W(E_7)$ can never change the $8$-height of a root in $\Phi$, so
it remains to prove that $\approx$ is a poset congruence on $\Omega$ by checking
conditions (i)--(iii) of Definition \ref{def:reading}.

Let $R\in \Omega$. By Lemma \ref{lem:e8roots} (ii), the $\approx$-equivalence class
of $R$ is given by 
\[
  [R]=\{R_0:=Q_1\cup F_0,\ R_1:=Q_1\cup F_1,\ R_2:=Q_1\cup F_2\},
\]
where $Q_1$ is the coplanar quadruple $Q_1=R\cap U_1$ and where $F_0, F_1$, and $F_2$ are
the alignment, crossing, and nesting in the $D_4$-subsystem $\Psi^\perp$
orthogonal to the $D_4$-subsystem $\Psi=\Span(Q_1)\cap \Phi$, respectively.
Furthermore, the elements $R_0, R_1, R_2$ have consecutive levels, and they form
an increasing chain in $\Omega$ by Lemma \ref{lem:poseven} (iii). In particular,
we have $R^-=R_0$ and $R^+=R_2$.

There exists a root $\al\in \Psi^\perp\se \Phi$ such that $s_\al(F_0)=F_2$. For
example, if we write $\al_1, \dots, \al_4$ for the simple roots of $\Psi^\perp$
as in Example \ref{exa:residues}, then we may take $\al=\al_1+\al_2$. Since
$\al$ is orthogonal to all roots in $\Psi$, the reflection $s_\al$ fixes all
elements of $Q_1$, so we also have $s_\al(R_0)=(R_2)$. Since $\rho(R_2) -
\rho(R_0) = 2$ is even, it then follows from \cite[Corollary 6.2]{rains13} that
the M\"obius function $\mu(R_0, R_2)$ is equal to $0$. This implies that the
open interval $(R_0, R_2)$ contains a single element, which must be $R_1$, so we
have $[R]=[R_0,R_2]=[R^-,R^+]$. This proves condition (i) of Definition
\ref{def:reading} for $\approx$.

It now suffices to prove that
if $R\leq R'$ for some $R'\in \Omega$, then we have
$R^-\le (R')^-$ and $R^+\le (R')^+$. To do so, it is enough to consider the
case where $R < R'$ is a covering relation, with $R' = rR >R$ for a reflection
$r=s_\al$ for some $\al\in \Phi(E_7)$. We will prove that $R^-\le R'^-$ and omit
the similar proof that $R^+\le R'^+$ in this case. 

Set $Q_0:=R\cap U_0=R\setminus Q_1$ as in Lemma \ref{lem:e8roots}, and let
$Q=\{\be\in R: B(\al,\be)\neq 0\}\se R$ as in Lemma \ref{lem:poseven}. Note that
the assumption that $rR>R$ implies $rQ > Q$ and $\rho(rQ)>\rho(Q)$ by Lemma
\ref{lem:poseven} (ii). 

If $Q = Q_0$, then $r$ fixes every root in $Q_1=R\cap U_1$, so we
have $R'\approx R$ and $R^- = (R')^-$.  If $Q \cap Q_0 = \emptyset$, then
$Q=Q_1$ and $r$ fixes every element in $Q_0$, so we have  
\[ R^-=Q\ \dot\cup \ F_0< (rQ)\ \dot\cup \ F_0=(R')^-, \] 
where the equalities follow from the second paragraph of this proof and the inequality follows
from the fact that $rQ > Q$ and Lemma \ref{lem:poseven}
(ii). We have proved that $R^-<R'^-$ if $Q$ is either $Q_0$ or $Q_1$.

If $Q$ is neither $Q_0$ nor $Q_1$, then Lemma \ref{lem:overlap} (i) implies that $H
= Q \cup Q_0$ has size $6$ and spans a subsystem $\Xi$ of type $D_6$, and that
the set $Q\cap Q_0$ is a couple in the $D_6$-subsystem
$\Xi=\Span(H)\cap \Phi$. Lemma \ref{lem:110000} (iv) then implies that we have $Q\cap
Q_1=\{\ep_1\pm\ep_j\}$ and $Q_0=\{\ep_a\pm \ep_b, \ep_c\pm\ep_d\}$ where
$\{\{1,j\}, \{a,b\}, \{c,d\}\}$ is a perfect matching of the set $[6]$. By Lemma
\ref{lem:fifteen} and its proof,
it follows that with respect to the poset congruence $\sim$ on $\Xi$, the minimum element $H^-$ in 
the $\sim$-equivalence class of $H$ is given by 
$H^-=(Q\cap Q_1)\ \dot\cup \ F$,
where $F$ is the alignment in the $D_4$-subsystem $\Span(Q_0)\cap \Phi$. This
system is orthogonal to $Q_1$ and thus precisely the system $\Psi^\perp$ from
the second paragraph by Lemma \ref{lem:e8roots} (ii), so we have $F=F_0$ and 
\begin{equation}
\label{eq:rh1}
R^-=H^-\ \dot\cup \ (R\setminus H).
\end{equation}
Since $r$ cannot change $8$-heights of roots, the set $r(H\cap Q_1)\se rH$
consists of a pair of roots of $8$-height 1 as well, and a similar argument shows that 
\begin{equation}
\label{eq:rh2}
  (rR)^-=(rH)^-\ \dot\cup\  (R\setminus H).
\end{equation}
Now, since $Q<rQ$, applying 
Lemma \ref{lem:poseven} (iii) to $\Xi$ shows that $H<rH$, and the fact that
$\sim$ is a poset congruence on $\Omega_{\Xi_+}$ then implies that $H^-\le
(rH)^-$. Therefore, there exist roots $\gamma_1,\gamma_2, \dots,
\gamma_p\in \Xi$ such that the reflections $r_i= s_{\gamma_i}$ give rise to an
increasing chain
\[H_0:=H^-< H_1:=r_1H^-< H_2:=r_2r_1H^-< \dots < H_p:=r_p\dots r_2r_1H^-=(rH)^- \] 
in $\Omega_{\Xi_+}$.
Since the roots $\gamma_i$ all lie in $\Xi$, the reflections $r_i$ all fix the two
roots in $R\setminus H$. 
Applying 
Lemma \ref{lem:poseven} to $\Xi$
proves that each $r_i$ acts on a quadruple $Q'_i \subset H_{i-1}$ 
satisfying $Q'_i < r_i Q'_i$; applying Lemma \ref{lem:poseven} again, to
$\Omega_{\Phi_+}$, then implies that
the same sequence of reflections produces a chain
\[R_0:=R^-< R_1:=r_1R^-< R_2:=r_2r_1R^-< \dots < R_p:=r_p\dots
r_2r_1R^-=(rR)^-=(R')^-.\] 
This completes the proof that $R^-\le (R')^-$ whenever $R<R'$, and we are done.
\end{proof}

\begin{cor}
\label{cor:630}
The set  $\Omega/\!\approx$ of equivalence classes of $\approx$  is a quasiparabolic
$W(E_7)$-set with level function $\lambda([R]):=\rho_{U}(R\cap U_1)$.
\end{cor}

\begin{proof} 
  Since $\approx$ is a $W(E_7)$-invariant poset congruence on $\Omega$, 
  Theorem \ref{thm:quotient} implies that
  $\Omega/\!\approx$ is a quasiparabolic $W(E_7)$-set with level function 
  \[
    \rho_{\Phi_+}^-([R])=\rho(R^-)=\rho_U(R\cap
  U_1)+1,
\]
where the last equality holds by
  Lemma \ref{lem:e8roots} (iv). 
  If $(X,\mu)$ is a quasiparabolic set for
  a given Coxeter group, then so is $(X,\mu-l)$ for any integer-valued constant function $l$ by
  Definition \ref{def:qpset}, so it follows that $\Omega/\!\approx$ is also a
  quasiparabolic $W(E_7)$-set with level function $\lambda$.
\end{proof}

\begin{theorem}\label{thm:e88}
Let $W$ be the Weyl group of type $E_8$, with generating set $S$ and root system
$\allroots$. Let $I = S \backslash \{s_8\}$, so that $W_I \cong W(E_7)$, and let
$U \subset \allroots$ be the set of all roots of $8$-height $1$. The triple
$T(E_8, 8)=(W, I, U)$ is a QP-triple, and the $W_I$-set $(\Omega_U,\rho_U)$ is
isomorphic to the quasiparabolic $W_I$-set $(\Omega/\!\approx,\lambda)$ from
Corollary \ref{cor:630}, in the sense that the map $\phi:\Omega/\!\approx\ \ra
\Omega_U, [R]\mapsto R\cap U$ is a well-defined $W_I$-equivariant bijection that preserves
level.  \end{theorem}

\begin{proof}
  The definition of $\approx$ and Corollary \ref{cor:complete} imply that
  $\Omega_U$ consists of coplanar quadruples in $U_1$ and that $\phi$ is a
  well-defined bijection, with the inverse of $\phi$ sending each quadruple
  $Q\in \Omega_U$ to the $\approx$-equivalence class in $\Omega$ consisting of
  the three $8$-roots $R$ such that $R\cap U=Q$. Moreover, the map $\phi$ is
  $W_I$-equivariant because elements of $W_I$ cannot change $8$-heights of roots
  in $\Phi$, and Corollary \ref{cor:630} shows that $(\Omega/\!\approx,\lambda)$ is a
  quasiparabolic $W_I$-set such that $\lambda([R])=\rho_U(\phi([R]))$ for all
  $R\in \Omega$. It follows that $\phi$ transports the quasiparabolic
  $W_I$-set structure of $(\Omega/\!\approx,\lambda)$ to $(\Omega_U, \rho_U)$, which
completes the proof. \end{proof}

\begin{rmk}
\label{rmk:e8min}
Recall from Remark \ref{rmk:nrtriples} (iii) that the $8$-root $$ R = \{\al_2,
\al_3,\al_5, \al_7, \theta_4, \theta_6, \theta_7, \theta=\theta_8\} $$ is the
unique $W(E_8)$-minimal element of $\Omega_{\Phi_+}$ and has level $\rho(R)=0$.
Theorem \ref{thm:nrtriples} implies that the element $s_8(R)$, which lies in
$\Omega$, has level $\rho(s_8(R))=1$, and it then follows from Lemma
\ref{lem:e8roots} (iii) that $\res(s_8(R))=\{\theta\}$. It can be shown (by
computation, for example) that the unique $W(E_7)$-minimal element of $\Omega_U$
is given by  $s_8(\theta)\cap U=\{s_8(\al_7), s_8(\theta_6), s_8(\theta_7),
s_8(\theta)\}$.
\end{rmk}

\section{The QP-triple \texorpdfstring{$T(E_7,7)$}{T(E7, 7)}}
\label{sec:45}
We focus on the QP-triple \[T(E_7, 7)=(W(E_7), S(E_7)\setminus\{s_7\},
U(E_7,7))\] in this section. The QP-triple $T(E_8,8)$ treated in Section
\ref{sec:e88} plays an important role in the section, as we will essentially be
able to embed the set $U'=U(E_7,7)$ into the set $U=U(E_8,8)$ and embed
$\Omega_{U'}$ into $\Omega_U$ (Lemma \ref{lem:iota}). In
Section \ref{sec:e77pf} (Theorem \ref{thm:e77}), we will use these embeddings to
deduce the fact that $T(E_7,7)$ is a QP-triple from Theorem \ref{thm:e88}. Both the sets of roots $U$ and
$U'$ also have natural connections to del Pezzo surfaces and minuscule
representations, and we extend these connections to the orthogonal sets
associated with $T(E_7,7)$ in sections \ref{sec:27} and \ref{sec:f3}. More
specifically, we give explicit descriptions for the generalized Rothe diagrams
of the elements of $\Omega(E_7, 7)$ in Section \ref{sec:27} (Theorem
\ref{thm:rothe45}), which we expect may have applications to del Pezzo surfaces,
and we show that generalized Pfaffian of $T(E_7, 7)$ coincides with the
invariant cubic polynomial associated to a certain 27-dimensional minuscule
representation of type $E_6$ in Section \ref{sec:f3} (Theorem
\ref{thm:vavilovcubic}).

\subsection{Quasiparabolic structure}
\label{sec:e77pf}
In this subsection, we maintain the notation of Theorem \ref{thm:e88}, we identify
the root system $\Phi(E_7)$ with the subsystem $\Phi'=\Phi_{8,0}$ of $\Phi$ in
the usual way, and we identify the set $U'=U(E_7,7)$ with the subset of $\Phi'_+$
consisting of roots of positive odd $7$-height. Let $\theta'=\theta_7$ be the
highest root of $\Phi'$. The $7$-height of $\theta'$ is 1 by Equation
\eqref{eq:theta7}, so it follows that $U'=\Phi'_{7,1}$. We will use Theorem
\ref{thm:e88} to prove that $T(E_7, 7)$ is a QP-triple in this subsection. In
other words, we will prove that $\Omega_{U'}$ is a quasiparabolic set for
$W_{I'}\cong W(E_6)$ with level function $\rho_{U'}$, where $I'=S \setminus
\{s_8, s_7\}$.  

We note that the set $U=U(E_8, 8)$ has a unique maximal element, which is also
the second highest root of $\Phi$, namely, $\eta:=s_8(\theta)=\theta-\al_8$.

\begin{lemma}\label{lem:e7orthres}
Let $Q \in \Omega_U$ and suppose that $\eta \in Q$.
\begin{itemize}
\item[{\rm (i)}]{We have $\al_8 \not\in \res_U(Q)$.}
\item[{\rm (ii)}]{Every element of $\res_U(Q)$ is orthogonal to $\eta$.}
\end{itemize}
\end{lemma}

\begin{proof}
We have $B(\al_8, \eta) = -1$, which implies that $\al_8 \not\in Q$. The
elements of $Q$ sum to $2\theta$, so we have $B(\al_8, 2\theta) = +2$. This can
only happen if every $\beta \in Q \backslash \{\eta\}$ satisfies $B(\al_8,
\beta) = +1$, which implies that $s_\beta(\al_8) < 0$ for all $\be\in
Q\setminus\{\eta\}$, proving (i).

Let $\gamma \in \res_U(Q)$. Remark \ref{rmk:negate} (i) implies that $\gamma\neq
\eta$, so we have $B(\eta,\gamma)\neq 2$. If it were the case that $B(\eta,
\gamma)=-1$, it would follow that $\eta + \gamma = \theta$ by considering
$8$-heights. This would imply that $\gamma = \al_8$, contradicting (i). If we
had $B(\eta, \gamma)=+1$, it would follow that $s_\eta(\gamma) < 0$, and
$\gamma$ would not dominate $Q$. We conclude that $B(\eta, \gamma) = 0$, and
(ii) follows. \end{proof}

\begin{lemma}\label{lem:iota}
\begin{itemize}
\item[{\rm (i)}]{ The simple reflection $s_8$ gives a bijection from $U'$ to the
  set \[\eta_U^\perp:=\{\be\in U: B(\be, \eta)=0\}\subseteq U.\]}  
\item[{\rm (ii)}]{Each element $T \in \Omega_\upr$ has size 3, so that
  $\Omega_{U'}$ consists of triples of pairwise orthogonal roots from $U'$.}
\item[{\rm (iii)}]{There is an injective $W(E_6)$-equivariant map $\iota :
  \Omega_\upr \ra \Omega_U$ given by $$ \iota(\{\be_1, \be_2, \be_3\}) =
\{s_8(\be_1), s_8(\be_2), s_8(\be_3), \eta\} .$$}
\item[{\rm (iv)}] For any distinct roots $\be, \gamma\in U'$, we have
  $s_\be(\gamma)<0$ if and only if $s_{s_8(\be)}(s_8(\gamma))<0$.
\item[{\rm (v)}]{For every $T\in \Omega_{U'}$, we have a bijection
$\res_\upr(T)\ra \res_U(\iota(T))$ sending each $\gamma\in \res_{U'} (T)$ to
$s_8(\gamma)$. In particular, we have $\lht_\upr(T) = \lht_U(\iota(T))$.}
\end{itemize} \end{lemma}

\begin{proof}
Let 
\[
  X_{i,j}=\{\al\in \Phi: \rootht_7(\al)=i\text{ and } \rootht_8(\al)=j\}
\]
for all $i,j\in \Z$. Direct computation shows that $s_8(X_{i,j})\se X_{i,i-j}$ and
 for all $i,j\ge 0$, and we have $B(\be,
\eta)=B(s_8(\be), \theta)=\rootht_{8}(s_8(\be))$ for any $\be\in \Phi$. It
follows that for any $\gamma\in U'=X_{1,0}$, we have $s_8(\gamma)\in X_{1,1}\se
U$ and $B(s_8(\gamma),\eta)=\rootht_8(\gamma)=0$, so $s_8(\gamma)\in
\eta_U^\perp$. It also follows that for any $\be\in \eta_U^\perp$ the element
$\gamma=s_8(\beta)$ has $8$-height 0, which implies that $\gamma\in X_{1,0}=U'$.
Part (i) follows.

There exist orthogonal triples in $\upr$, such as the triple $\{s_8(\al_7),
s_8(\theta_6), s_8(\theta_7)\}$ (see Remark \ref{rmk:e8min}). For every
$\be \in U'$, the root $s_8(\beta)$ has $8$-height 1 and 7-height 1, so (i)
implies that if $T \in \Omega_\upr$ then we have $s_8(T) \cup \{\eta\}\in
\Omega_U$. The set $s_8(T)\cup \{\eta\}$ has size 4 in this case by Corollary
\ref{cor:complete}, so $T$ has size 3, proving (ii). 

The last paragraph shows that the map $\iota$ given in (iii) sends elements of
$\Omega_{U'}$ to $\Omega_U$. The map $\iota$ is injective because
$T=s_8(\iota(T)\setminus\{\eta\})$ for all $T\in \Omega_{U'}$, and it is
$W(E_6)$-equivariant because every simple root in $S\setminus\{\al_8,\al_7\}$ is
orthogonal to both $\al_8$ and $\eta$. Part (iii) follows.

For any distinct roots $\be, \gamma\in U'$, we have \[
s_{s_8(\beta)}(s_8(\gamma))=s_8(\gamma)-B(s_8(\be),s_8(\gamma))s_8(\beta)
=(\gamma+\al_8)-B(\be,\gamma)(\beta+\al_8). \] It follows that
$s_{\be}(\gamma)<0$ if and only if $s_{s_8(\be)}(s_8(\gamma))<0$, because both
conditions hold if and only if we have both $B(\be,\gamma)=1$ and $\gamma<\be$.
This proves (iv). 

Let $T = \{\be_1, \be_2, \be_3\}\in\Omega_{U'}$ and $\gamma\in \res_{U'}(T)$.
Since $B(\eta, s_8(\gamma))=B(\theta,\gamma)=0$, we have
$s_{s_8(\gamma)}(\eta)=\eta>0$. For each $i\in \{1,2,3\}$, we have $\gamma\neq
\be_i$ by Remark \ref{rmk:negate} (i), and we have $s_{\be_i}(\gamma)>0$ since
$\gamma\in \res_{U'}(T)$, so (iv) implies that $s_{s_8(\be_i)}(s_8(\gamma))>0$.
It follows that $s_8(\gamma)\in \res_{U}(\iota(T))$. 

Conversely, for any $\delta\in \res_U(\iota(T))$, we have $\delta\in
\eta_U^\perp$ by Lemma \ref{lem:e7orthres} (ii), so we have $\delta=s_8(\gamma)$
for some $\gamma\in U'$ by (i). The assumption that $\delta\in \res_U(\iota(T))$
implies that $\delta\notin \iota(T)$ and, therefore, that $\gamma\notin T$.
Since $\delta=s_8(\gamma)\in \res_U(\iota(T))$, it follows from (iv) that
$\gamma\in \res_{U'}(T)$, which proves the first assertion of (v). The second
assertion follows from the first by taking cardinalities. \end{proof}

\begin{theorem}\label{thm:e77} Let $W$ be the Weyl group of type $E_8$, with
  generating set $S$ and root system $\allroots$. Let $I' = S \backslash \{s_7,
  s_8\}$, so that $W_{I'} \cong W(E_6)$, and let $\upr \subset \allroots$ be the
  set of all roots of $8$-height $0$ and $7$-height $1$. The triple $T(E_7,
7)=(W, I', U')$ is a QP-triple; in other words, the set $\Omega_\upr$ is a
quasiparabolic $W(E_6)$-set with respect to the level function $\lht_\upr$.
\end{theorem}

\begin{proof}
The set $\Omega_{U'}$ is a $W_{I'}$-set because elements of $I$ do not change
$7$-heights of roots. Lemma \ref{lem:iota} (iii) and (v) imply that for
any $T\in \Omega_{U'}$ and any reflection $r\in W_{I'}$, we have
$\rho_{U'}(T)=\rho_U(\iota(T))$ and
$\rho_{U'}(rT)=\rho_U(\iota(rT))=\rho_U(r(\iota(T))$ for the element
$\iota(T)\in \Omega_U$. The fact that $\Omega_\upr$ is a quasiparabolic
$W(E_6)$-set with respect to $\rho_\upr$ now follows from Theorem \ref{thm:e88}
and the injectivity of $\iota$ by a routine verification. \end{proof}

\begin{rmk}
\label{rmk:e7min}
\begin{enumerate}
\item There is a $W(E_6)$-invariant poset congruence on the $135$ positive
$7$-roots in type $E_7$, in which two $7$-roots are equivalent if they agree
modulo their $7$-height $0$ components. The corresponding quotient quasiparabolic
$W(E_6)$-set is canonically isomorphic to $\Omega_\upr$, analogously to
Theorem \ref{thm:e88}.
  \item Combining Remark \ref{rmk:e8min} and Lemma \ref{lem:iota}
    shows that the unique $W(E_6)$-minimal element of $\Omega_\upr$ is the
    triple $\{\al_7, \theta_6, \theta_7\}$.
\end{enumerate}
\end{rmk}

\subsection{Description of generalized Rothe diagrams}
\label{sec:27}
Let $U=U(E_8,8)$ and $U'=U(E_7,7)$ as in Section \ref{sec:e77pf}. We describe
the generalized Rothe diagrams of the elements of $\Omega_{U'}$ in this
subsection (Theorem \ref{thm:rothe45}). The description will be given in notation motivated by the algebraic geometry of del Pezzo surfaces, and we
expect that the generalized Rothe diagrams, which are new to our knowledge, may
have applications to Schubert varieties or del Pezzo surfaces. 

The sets $U$ and $U'$ also have the following natural connections to algebraic
geometry. The roots in $U$ are in bijection with the 56 exceptional curves on a
del Pezzo surface $\mathcal{S}$ of degree 2, and the roots in $U'$ are in
bijection with the 27 lines on a del Pezzo surface $\mathcal{S}'$ of degree 3
\cite[Sections 9.3 and 10.1]{green13}. Under these bijections, two roots are
orthogonal if and only if the intersection number between their corresponding
curves is equal to $1$, so the triples of $\Omega_{U'}$ are in a bijective
correspondence with the tritangent planes of $\mathcal{S}'$. The quadruples of
$\Omega_U$ correspond bijectively to the what are sometimes called ``tetrads''
on $\mathcal{S}$; see, for example, \cite[Proposition 5.4]{gizatullin18}. We
note that theorems \ref{thm:e88} and \ref{thm:e77} imply that the tritangent
planes and tetrads carry natural partial orders induced by the quasiparabolic
orders of $\Omega_{U'}$ and $\Omega_U$, respectively. The existence of these
induced orders is new to the best of our knowledge.  

To describe the generalized Rothe diagrams of the elements of $\Omega_{U'}$, we
will now use the following labelling conventions for the roots of $U$ adapted
from the context of del Pezzo surfaces. We note that even though we will be
focusing on $\Omega_{U'}$, it is more convenient to explain these conventions
via the set $U$. 

Let $\{\ep_1,\ldots,\ep_8\}$ be an orthonormal basis for $\R^8$ as usual, and
let $\{\al_1, \ldots, \al_8\}$ be the simple roots of $\Phi=\Phi(E_8)$, so that
each element $\be\in U$ has the form $\be=\al_8+\sum_{i=1}^7 d_i\al_i$ for some
$d_1,\ldots, d_7\ge 0$. In the rest of the paper, we will adopt the labelling
convention for the Dynkin diagram of type $E_7$ shown in Figure \ref{fig:fanoe},
which is compatible with the realization of the root system $\allroots(E_7)$
inside $\R^8$ where the simple roots are $\al'_i := 4(\ep_i - \ep_{i+1})$ for $1
\leq i \leq 6$ and \[ \al'_7 := (-2, -2, -2, 2, 2, 2, 2, -2). \] In this
realization, the $126$ roots of type $E_7$ consist of the $56$ coordinate
permutations of $4(\ep_1 - \ep_2)$ and the $70$ coordinate permutations of
$\al'_7$; a root is positive if and only if its rightmost nonzero coefficient is
negative. More details about this realization can be found in \cite[Example
8.1.7]{green13}.

\begin{figure}[!ht]
  \begin{center} 
\begin{tikzpicture} \node[main node] (1) {}; \node[main node] (2) [right=1cm of
      1] {}; \node[main node] (3) [right=1cm of 2] {}; \node[main node] (4)
      [right=1cm of 3] {}; \node[main node] (5) [right=1cm of 4] {}; \node[main
      node] (6) [right=1cm of 5] {}; \node[main node] (7) [above=1cm of 3] {};

            \path[draw]
            (1)--(2)--(3)--(4)--(5)--(6)            
            (3)--(7);

            \node (11) [below=0.1cm of 1] {\small{$1$}};
            \node (22) [below=0.1cm of 2] {\small{$2$}};
            \node (33) [below=0.1cm of 3] {\small{$3$}};
            \node (44) [below=0.1cm of 4] {\small{$4$}};
            \node (55) [below=0.1cm of 5] {\small{$5$}};
            \node (66) [below=0.1cm of 6] {\small{$6$}};
            \node (77) [above=0.1cm of 7] {\small{$7$}};
\end{tikzpicture}
  \end{center}
\caption{New labelling of the Dynkin diagram of type $E_7$}
\label{fig:fanoe}
\end{figure}

For any root $\be=\al_8+\sum_{i=1}^7 d_i\al_i\in U$, we define 
\[
  f(\beta)=(-1,-1,-1,-1,-1,-1,3,3)+\sum_{i=1}^7 d_i\al'_{\tau(i)},
\]
where $\tau(i)$ refers to the same vertex labelled by $i$ in the Dynkin diagram
in the convention of Figure \ref{fig:ade} (d). It is known that the map $f$ is a
bijection between $U$ and the set of vectors $\{\pm v_{i,j}: 1\le i<j\le 8\}$,
where $v_{i,j}=4(\ep_i+\ep_j) -\sum_{k=1}^8 \ep_k$ for all $i,j$ (\cite[Example
8.1.7]{green13}). We will write $X_{ij}$ for $v_{i, j}$ and $Y_{ij}$ for $-v_{i,
j}$. We will also identify each vector $X_{ij}$ or $Y_{ij}$ with the
corresponding root in $U$ from now on. 

Let $a_i=X_{i7}$, $b_i=X_{i8}$, and $c_{ij}=c_{ji}=Y_{ij}$ for all $1 \leq i
\ne j \leq 6$. Routine verification shows that we have  
\[
  U'=\{a_i: 1\le i\le 6\}\dcup \{b_i: 1\le i\le6\}\dcup \{c_{ij}: 1\le i<j\le
  6\},
\]
and that three elements of $\upr$ form a triple in $\Omega_\upr$ if and only if
they are of the form $a_i c_{ij} b_j$ for some $1 \leq i \ne j \leq 6$ or of the
form $c_{i_1j_1}c_{i_2j_2}c_{i_3,j_3}$ corresponding to a perfect matching
$\{\{i_1,j_1\},\{i_2,j_2\},\{i_3,j_3\}\}$ of $\{1,2,\ldots,6\}$. 

\begin{exa} \label{exa:56} 
Recall from Remark \ref{rmk:e7min} (ii) that the unique $W(E_6)$-minimal element
in $\Omega_{U'}$ is the triple $x_0:=\{\al_7, \theta_6, \theta_7\}$ (written in
the labelling convention of Figure \ref{fig:ade} (d)). Direct translation of
notation shows that we have $x_0=\{b_6, c_{16}, a_{1}\}$. 
\end{exa}

\begin{lemma}\label{lem:madeneg}
If $\be$ and $\gamma$ are distinct elements of $\upr$, then we have
$s_\gamma(\be)<0$ if and only if one of the following conditions holds:
\begin{itemize}
\item[{\rm (i)}]{$\be = a_i$ and $\gamma = a_j$ for some $j < i$;}
\item[{\rm (ii)}]{$\be = b_i$ and $\gamma = a_j$ for some $i \ne j$;}
\item[{\rm (iii)}]{$\be = b_i$ and $\gamma = c_{jk}$, where $i$, $j$, and $k$ are
distinct;}
\item[{\rm (iv)}]{$\be = c_{ij}$ and $\gamma = c_{kj}$, where $k > i < j$ and $k \ne j$;}
\item[{\rm (v)}]{$\be = c_{ij}$ and $\gamma = c_{il}$ for any $i < j < l$;}
\item[{\rm (vi)}]{$\be = c_{ij}$ and $\gamma = a_h$, where $h$, $i$, and $j$ are 
distinct.}
\end{itemize}
\end{lemma}

\begin{proof}
This follows from a routine case by case check using the earlier description of
the positive roots of type $E_7$ in the new coordinate realization.
\end{proof}

\begin{theorem}\label{thm:rothe45}
The generalized Rothe diagrams of the triples in $\Omega_{U'}$ are given as follows.
\begin{itemize}
\item[{\rm (i)}]{If $1 < i < j < 6$, then the triple $a_i c_{ij} b_j$ has level
  $5 + i - j$, and we have $$\res_{U'}(a_ic_{ij}b_j)=\{a_h : 1 \leq h < i\} \
\cup \ \{c_{ik} : j < k \leq 6\} .$$}
\item[{\rm (ii)}]{If $1 < j < i < 6$, then the triple $a_i c_{ij} b_j$ has level
  $4 + i - j$, and we have $$ \res_{U'}(a_ic_{ij}b_j)= \{a_h : 1 \leq h < i\} \
\cup \ \{c_{ki} : j < k \leq 6 \text{\ and\ } k \ne i\} .$$}
\item[{\rm (iii)}]{If $M:=\{\{i_1,j_1\},\{i_2,j_2\},\{i_3,j_3\}\}$ is a perfect
    matching of $[6]$, then the triple $c_{i_1j_1}c_{i_2j_2}c_{i_3j_3}$ has
    level $(\ell(\tau)+9)/2$, where $\tau$ is the fixed-point-free involution
    associated to $M$ and $\ell(\tau)$ is the length of $\tau$. We have
$$\res_{U'}(c_{i_1j_1}c_{i_2j_2}c_{i_3j_3})= \{a_i : 1 \leq i \leq 6\} \ \cup \
\{c_{pq}: \tau(p) < q \text{\ and\ } \tau(q) < p\}. $$} \end{itemize}
\end{theorem}

\begin{proof}
For each triple $T\in \Omega_{U'}$, we have 
\[
  \res_{U'}(T)=\{\be\in U': s_\gamma (\be)>0 \text{ for all } \gamma\in
    T\}=\bigcap_{\gamma\in T}(U'\setminus\{\be\in U': s_\gamma(\be)<0\}), 
\]
where each set appearing in the intersection can be calculated using Lemma
\ref{lem:madeneg}. The descriptions of the generalized Rothe diagrams follow
from the lemma by straightforward computations. Taking cardinalities of the
generalized Rothe diagrams yields the levels of the triples in (i) and (ii), and
the claim about the level in (iii) follows from Theorem \ref{thm:hafpfaf} (v).
\end{proof}

\begin{exa}
\label{exa:45res}
Figure \ref{fig:45res} depicts the generalized Rothe diagrams of the elements
$T_1=a_4c_{24}b_2$ and $T_2=c_{13}c_{25}c_{46}$, in the following way.  First,
we identify the 27 roots of $U'$ with the 27 square cells in the array shown in
Figure \ref{fig:45res} (a) as follows: identify the cells in the top row with
$a_1, a_2, \ldots,$ and $a_6$ from left to right, the cells in the rightmost
column with $b_1, b_2, \ldots,$ and $b_6$ from top to bottom, and the unique
cell in the $(a_j, b_i)$-position (i.e., the cell in the same row as $b_i$ and
same column as $a_j$) with $c_{ij}=c_{ji}$. If we represent each $T_i$ by
marking the elements of $T$ with solid black dots, then the roots dominating
$T_i$ are the cells marked with a star.  For an element of the form
$T=a_ic_{ij}b_j$ where $j<i$, such as $T_1$, Theorem \ref{thm:rothe45} (ii)
implies that $\res_{U'}(T)$ consists of the cells to the left of $a_i$, or below
$c_{ji}$, or to the left of $b_i$ (which is below $b_j$ and not in $T$).  For an
element of the form $T=c_{i_1j_1}c_{i_2j_2}c_{i_3j_3}\in \Omega_{U'}$, such as
$T_2$, one can show using Theorem \ref{thm:rothe45} (iii) that $\res_{U'}(T)$
consists of all the six roots in the top row and the roots $c_{ij}$ in the
27-cell grid that are neither attacked from the right nor from below by any of
the roots in $T$ or their mirror images across the diagonal consisting of the
$(a_k,b_k)$-positions for $k\in \{1,2,\ldots,6\}$. We have shaded the cell
$b_i=b_4$ in Figure \ref{fig:45res} (a) and the aforementioned mirror images in
Figure \ref{fig:45res} (b).

\begin{figure}[h!]
\centering

\subfloat[{$\res_{U'}(a_4c_{24}b_2)$}]{
\begin{tikzpicture}[scale=0.7]

\foreach \x in {0,...,5}
    \draw (\x,6) rectangle (\x+1,7);

\foreach \x in {1,...,6}
    \draw (\x,5) rectangle (\x+1,6);

    \node at (0.5,7.3) {$1$};
    \node at (1.5,7.3) {$2$};
    \node at (2.5,7.3) {$3$};
    \node at (3.5,7.3) {$4$};
    \node at (4.5,7.3) {$5$};
    \node at (5.5,7.3) {$6$};
    \node at (7.3,0.5) {$6$};
    \node at (7.3,1.5) {$5$};
    \node at (7.3,2.5) {$4$};
    \node at (7.3,3.5) {$3$};
    \node at (7.3,4.5) {$2$};
    \node at (7.3,5.5) {$1$};

\foreach \x in {2,...,6}
    \draw (\x,4) rectangle (\x+1,6);

\foreach \x in {3,...,6}
    \draw (\x,3) rectangle (\x+1,6);

\foreach \x in {4,...,6}
    \draw (\x,2) rectangle (\x+1,6);

\foreach \x in {5,6}
    \draw (\x,1) rectangle (\x+1,6);

\foreach \x in {6}
    \draw (\x,0) rectangle (\x+1,6);

\node at (0.5, 6.5) {$\bigstar$};
\node at (1.5, 6.5) {$\bigstar$};
\node at (2.5, 6.5) {$\bigstar$};

\filldraw[black] (3.5, 6.5) circle (4pt); 
\filldraw[black] (3.5, 4.5) circle (4pt); 
\filldraw[black] (6.5, 4.5) circle (4pt); 

\filldraw[gray!30] (6.5, 2.5) circle (4pt); 

\node at (3.5, 3.5) {$\bigstar$};
\node at (4.5, 2.5) {$\bigstar$};
\node at (5.5, 2.5) {$\bigstar$};
      
\end{tikzpicture}}
\quad\quad\quad\quad
\subfloat[{$\res_{U'}(c_{13}c_{25}c_{46})$}]{
\begin{tikzpicture}[scale=0.7]

    \node at (0.5,7.3) {$1$};
    \node at (1.5,7.3) {$2$};
    \node at (2.5,7.3) {$3$};
    \node at (3.5,7.3) {$4$};
    \node at (4.5,7.3) {$5$};
    \node at (5.5,7.3) {$6$};
    \node at (7.3,0.5) {$6$};
    \node at (7.3,1.5) {$5$};
    \node at (7.3,2.5) {$4$};
    \node at (7.3,3.5) {$3$};
    \node at (7.3,4.5) {$2$};
    \node at (7.3,5.5) {$1$};

\node at (0.5, 6.5) {$\bigstar$};
\node at (1.5, 6.5) {$\bigstar$};
\node at (2.5, 6.5) {$\bigstar$};
\node at (3.5, 6.5) {$\bigstar$};
\node at (4.5, 6.5) {$\bigstar$};
\node at (5.5, 6.5) {$\bigstar$};

\foreach \x in {0,...,5}
    \draw (\x,6) rectangle (\x+1,7);

\foreach \x in {1,...,6}
    \draw (\x,5) rectangle (\x+1,6);

\foreach \x in {2,...,6}
    \draw (\x,4) rectangle (\x+1,6);

\foreach \x in {3,...,6}
    \draw (\x,3) rectangle (\x+1,6);

\foreach \x in {4,...,6}
    \draw (\x,2) rectangle (\x+1,6);

\foreach \x in {5,6}
    \draw (\x,1) rectangle (\x+1,6);

\foreach \x in {6}
    \draw (\x,0) rectangle (\x+1,6);

\foreach \x in {0,...,5}
    \foreach \y in {0,...,\x}
        \draw[dashed] (5-\x,\y)--(5-\x+1,\y);

\foreach \x in {0,...,5}
    \foreach \y in {0,...,\x}
        \draw[dashed] (5-\x,\y)--(5-\x,\y+1);
      
    \node at (5.5, 1.5) {$\bigstar$};
    
    \node at (4.5, 3.5) {$\bigstar$};

    \filldraw[gray!30] (0.5, 3.5) circle (4pt); 
    \filldraw[gray!30] (1.5, 1.5) circle (4pt); 
    \filldraw[gray!30] (3.5, 0.5) circle (4pt); 
    
    \filldraw[black] (2.5, 5.5) circle (4pt); 
    \filldraw[black] (4.5, 4.5) circle (4pt); 
    \filldraw[black] (5.5, 2.5) circle (4pt); 

\end{tikzpicture}}

\caption{The generalized Rothe diagrams of two elements in $\Omega_{U'}$}
\label{fig:45res}
\end{figure}

\end{exa}

Theorem \ref{thm:rothe45} allows us to compute $\qhf(U')$ explicitly. We have

\begin{align*}
\qhf(U')=
&\hphantom{+} q^{0}a_{1}c_{16}b_{6}+q^{1}a_{2}c_{26}b_{6}+q^{1}a_{1}c_{15}b_{5}+q^{2}a_{2}c_{25}b_{5}+q^{2}a_{3}c_{36}b_{6}+q^{2}a_{1}c_{14}b_{4}+q^{3}a_{2}c_{24}b_{4}\\
&+q^{3}a_{3}c_{35}b_{5}+q^{3}a_{4}c_{46}b_{6}+q^{3}a_{1}c_{13}b_{3}+q^{4}a_{2}c_{23}b_{3}+q^{4}a_{3}c_{34}b_{4}+q^{4}a_{4}c_{45}b_{5}+q^{4}a_{5}c_{56}b_{6}\\
&+q^{4}a_{1}c_{12}b_{2}+q^{5}b_{1}c_{12}a_{2}+q^{5}b_{2}c_{23}a_{3}+q^{5}b_{3}c_{34}a_{4}+q^{5}b_{4}c_{45}a_{5}+q^{5}b_{5}c_{56}a_{6}+q^{6}b_{1}c_{13}a_{3}\\
&+q^{6}b_{2}c_{24}a_{4}+q^{6}c_{12}c_{34}c_{56}+q^{6}b_{3}c_{35}a_{5}+q^{6}b_{4}c_{46}a_{6}+q^{7}b_{1}c_{14}a_{4}+q^{7}c_{13}c_{24}c_{56}+q^{7}b_{2}c_{25}a_{5}\\
&+q^{7}c_{12}c_{35}c_{46}+q^{7}b_{3}c_{36}a_{6}+q^{8}b_{1}c_{15}a_{5}+q^{8}c_{14}c_{23}c_{56}+q^{8}c_{13}c_{25}c_{46}+q^{8}b_{2}c_{26}a_{6}+q^{8}c_{12}c_{36}c_{45}\\
&+q^{9}b_{1}c_{16}a_{6}+q^{9}c_{15}c_{23}c_{46}+q^{9}c_{14}c_{25}c_{36}+q^{9}c_{13}c_{26}c_{45}+q^{10}c_{16}c_{23}c_{45}+q^{10}c_{15}c_{24}c_{36}\\
&+q^{10}c_{14}c_{26}c_{35}+q^{11}c_{16}c_{24}c_{35}+q^{11}c_{15}c_{26}c_{34}+q^{12}c_{16}c_{25}c_{34}.
\end{align*}

\subsection{Invariant cubic polynomial}
\label{sec:f3}
We now show that the generalized Pfaffian $\pf({U'})=\qhf(U',-1)$ associated to
the QP-triple $T(E_7, 7)$ agrees with a well-known invariant cubic polynomial
$F$ associated with the minuscule representation $L(E_6, \omega_1)$ of the Lie
algebra (or Chevalley group) of type $E_6$. The polynomial $F$ is a signed sum
of 45 monomials corresponding to triples of weights of $L(E_6, \omega_1)$ that
are called ``triads" and that can be identified with the 45 elements of
$\Omega_{U'}$, as we explain in the next paragraph. The study of $F$ dates back
to Dickson \cite{Dickson1901}, who introduced $F$ in terms of the 27
lines on the del Pezzo surface mentioned in Section \ref{sec:27}, and the
explicit computation of $F$ in terms of triads was previously done
recursively using \texttt{Mathematica} by
Vavilov--Luzgarev--Pevzner \cite{vavilov07}. The equation $F=\pf(U')$ is new,
and it is remarkable in that it provides a non-recursive combinatorial
interpretation of the coefficients of the polynomial in terms of the parity of
the cardinalities of the associated generalized Rothe diagrams (which can be
computed manually using Theorem \ref{thm:rothe45}).

To relate $U'$ to $F$, recall that the exceptional Lie algebras of types $E_6$,
$E_7$, and $E_8$ have three minuscule representations between them: two
$27$-dimensional representations in type $E_6$, $L(E_6, \omega_1)$ and $L(E_6,
\omega_6)$, and a $56$-dimensional representation in type $E_7$, $L(E_7,
\omega_7)$, where each $\omega_i$ indicates the fundamental weight that appears
as the highest weight of the corresponding representation (see, for example,
\cite[Theorem 5.1.5]{green13}). If one restricts the adjoint representation of
the simple Lie algebra $\fg(E_8)$ to $\fg(E_7)$, then the set $U$ corresponds to
the weights of $L(E_7, \omega_7)$ \cite[Proposition 10.5.5]{green13}. Similarly,
restricting the adjoint representation of $\fg(E_7)$ to $\fg(E_6)$ identifies
$U'$ with the set of weights of $L(E_6, \omega_1)$. A triple of weights of
$L(E_6, \omega_1)$ forms a \emph{triad} if they sum to zero, and under the above
identification, triads are precisely triples of pairwise orthogonal elements of
$U'$ \cite[Section 12]{vavilov07}. The monomials in $F$ are given by the
products of triads, and we will write these monomials in the form
$x_{T}:=\al\be\gamma$ where $T=\{\al,\be,\gamma\}\in\Omega_{U'}$.

\begin{theorem}\label{thm:vavilovcubic}
The generalized Pfaffian $\pf(U')$ associated to the QP-triple $T(E_7, 7)$ is
the invariant cubic polynomial $F$ constructed in Vavilov--Luzgarev--Pevzner
\cite[Section 12]{vavilov07}. \end{theorem}

\begin{proof}
  Recall from Example \ref{exa:56} that the unique minimal $W(E_6)$-element of
  $\Omega_{U'}$ is the triple $x_0=\{b_6, c_{16}, a_1\}$, which has level
  $\lambda(x_0)=0$ by
  Theorem \ref{thm:rothe45} (i). In \cite[\S12]{vavilov07},
  Vavilov--Luzgarev--Pevzner denote $x_0$ as $(\lambda_0, \mu_0, \nu_0)$ and
  call it the ``distinguished triad". They define the invariant cubic
  polynomial $F$ to be the linear combination of the $45$ elements in the
  $W(E_6)$-orbit of $x_0$ in which the coefficient of any monomial $x_T$ is
  $(-1)^{\ell(w)}$, where $w$ is an element of minimum length such that $w(x_0)
  = x_T$. As recalled in Section \ref{sec:rqs}, this length is precisely
  $\lambda(x_T)-\lambda(x_0)=\lambda(x_T)$ by general properties of
  quasiparabolic sets, so we have $F=\pf({U'})$. 
\end{proof}

\begin{rmk}\label{rmk:vavilov}
\begin{itemize}
\item[(i)]{Vavilov--Luzgarev--Pevzner note in \cite[\S12]{vavilov07} that if we
    apply a simple reflection $s_i$ to one of the monomials $x_T$ in the
    invariant cubic polynomial, then either $s_i$ fixes $x_T$, or the
    coefficients of $x_T$ and $s_i(x_T)$ are the negatives of each other. This
  property also characterizes $\pf(U')$ up to scaling and provides an
alternative proof of Theorem \ref{thm:vavilovcubic}.}

\item[(ii)]{The cubic polynomial $F$ is written down explicitly in \cite[Table
  11]{vavilov07}. This makes it possible, after translating notation, to verify
Theorem \ref{thm:vavilovcubic} by hand.}
\item[(iii)]{Let $\Gamma$ be the graph whose vertices are the elements of
    $\Omega_\upr$, and with adjacency given by disjointness. It can be shown
    that $\Gamma$ is the complement of the graph $U_4(3)$ described in
    \cite[\S10.17]{brouwer22}. In particular, $\Gamma$ is strongly regular with
parameters $(45, 32, 22, 24)$.} \end{itemize} \end{rmk}

\begin{rmk}
\label{rmk:bitangent} It is possible to establish an analogue of Theorem
\ref{thm:vavilovcubic} for $\Omega_U$ as follows. The 56 vectors $X_{ij}$ and
$Y_{ij}$ arising from the set $U$ can be grouped into 28 pairs $[i,j]:=\{X_{ij},
Y_{ij}\}$. These pairs correspond naturally to the 28 bitangents to a plane
quartic curve, and the 630 quadruples of $\Omega_U$ correspond to the 315
``tetrads of bitangents" of the quartic curve \cite[Remark 5.5]{gizatullin18}.
One can modify the generalized Pfaffian of $\Omega_{U}$ in a precise way to
produce an invariant quartic polynomial known as Dickson's form (see
\cite[Definition 5.3]{gizatullin18} and \cite{dickson}), thus providing the
promised analogue of Theorem \ref{thm:vavilovcubic} for $\Omega_U$. We have
chosen not to describe the modification here because it requires elaborate
rescalings of coefficients for which we have no conceptual explanation. However,
we note that the study of geometric objects such as the 630 tetrads
corresponding to $\Omega_U$ and the 315 tetrads of bitangents has a long history
dating back to the works of Cayley and Salmon (see, for example, \cite[Remark
5.5]{gizatullin18} and \cite[Section 262]{salmon}). It would be interesting to
know if the set $\Omega_{U}$ has further implications for these classical
objects.  \end{rmk}

\section{Concluding remarks}\label{sec:conclusion}

\begin{rmk}
\label{rmk:matroid}
For each QP-triple $T=(W,I,U)$ from Table \ref{table:nrs_numbers} or Table
\ref{table:tnp_numbers}, it turns out that every $R\in \Omega_U$ gives rise to a
matroid, $M_R$, whose ground set is $R$ and whose set of flats is given by the
set \begin{equation} \label{eq:flats} \{R\cap \left(\bigcap_{R'\in Y}R'\right):
Y\se \Omega_U\}. \end{equation} The isomorphism type of $M_R$, as given in the
last columns of tables \ref{table:nrs_numbers} and \ref{table:tnp_numbers} (we
will explain the notation in the next paragraph), depends only on the triple $T$
and not on $R$. Furthermore, one can use the matroid $M_R$ to predict the
degrees of $U$ in all cases, as follows: if we list the possible sizes
for the flats of $M_R$ as $k_0, k_1, ..., k_r$, in increasing order, then each flat of size $k_i$ is a subset of precisely $d_i$ elements
of $\Omega_U$ for some number $d_i$ that depends only on $i$ and not on the
specific choice of the flat, and the sequence $(d_0/d_1, d_1/d_2,
\ldots, d_{r-1}/d_r)$ produces the sequence of degrees of the quasiparabolic set; see
Example \ref{exa:matroid}. 

As is the case in examples \ref{exa:aiii}--\ref{exa:diii}, for the QP-triples of
types $A$ or $D$, the assertions from the last paragraph can be checked using
relatively straightforward combinatorial arguments. We have also
verified the assertions for the QP-triples of type $E$ computationally. It
would be interesting to have uniform conceptual proofs of the assertions for all
the QP-triples listed in tables \ref{table:nrs_numbers} and
\ref{table:tnp_numbers}. In particular, it is intriguing that the matroids
can predict the degrees of $U$ via the quotients $d_i/d_{i+1}$ in all the
examples, because the degrees arise from the level generating function of
$\Omega_U$ but the definition of the matroids $M_R$ does not directly involve
generalized Rothe diagrams at all. It would also be interesting to know whether the
intersections in the set given by \eqref{eq:flats} have any inherent connections
with generalized Rothe diagrams. 

We now explain the notation in the last columns of tables
\ref{table:nrs_numbers} and \ref{table:tnp_numbers}. For nonnegative integers $j\le k$,
we use $U_{j, k}$ to denote the {uniform matroid} with ground set $[k]$ and rank
$j$, where a subset of the ground set is independent if and only if it has size
at most $j$. The notation $U^{(2)}_{j,k}$ stands for the doubling of $U_{j,k}$,
where the ground set and flats are given by the doublings of those of $U_{j,k}$.
Here, by a doubling of a set $A\se [k]$ we mean the set $A\cup\{a':a\in A\}$,
where $1,1',2,2',..., k,k'$ are understood to be $2k$ distinct elements. For
each positive integer $n$, the matroid $AG(n, 2)$ is the $n$-dimensional affine
space over $\FF_2$, where the nonempty flats are given by the affine subspaces.
Finally, the matroid $PG(n, 2)$ is the $n$-dimensional projective space over
$\FF_2$, with the nonempty flats given by the projective subspaces. All of these
matroids are examples of {binary matroids}, in the sense that they are
representable over the field $\FF_2$. More details about $U_{j,k}, AG(n,2),
PG(n,2),$ and binary matroids can be found in sections 1.2, 5.1, 5.4, and 6.3 of
\cite{matroids}, respectively. \end{rmk}

\begin{exa}
\label{exa:matroid}
For the triple $(W, S, \Phi_+)$ of type $E_8$ and for each $8$-root $R\in
\Omega_{\Phi_+}$, the possible sizes for the flats of $M_R$ are $0, 1, 2, 4$ and
$8$, and each flat of these sizes is contained in precisely $d_0=2025, d_1=135,
d_2=15, d_3=3,$ and $d_4=1$ element(s) of $\Omega_{\Phi_+}$, respectively; it is
possible to deduce the numbers $d_1, d_2, d_3$, and $d_4$ from Remark
\ref{rmk:complement}. The quotients $d_i/d_{i+1}$ are $2025/135=15$,
$135/15=9$, $15/3=5$, and $3/1=3$, which agree with the degrees of $U$ by Remark
\ref{rmk:nrtriples} (i). \end{exa}

\begin{rmk}
\label{rmk:symmspace}
The QP-triples $T(\Gamma,p)$ described in Table
\ref{table:tnp_numbers} have intimate connections with symmetric spaces.
Recall that symmetric spaces were classified by Cartan, who reduced the
classification to that of real simple Lie algebras (\cite{cartan27}, \cite[Chapter
X]{helgason}). To each irreducible symmetric space $G/K$ one may associate a
so-called symmetric pair $(\fg,\fk)$, which is a pair of real semisimple Lie
algebras satisfying a Cartan decomposition $\fg=\fk\oplus \fp$, where $\fk$ is
the fixed-point subalgebra of a Lie algebra involution $\theta$ and $\fp$ is the
$(-1)$-eigenspace of $\theta$. The pair $(\fg,\fk)$ is not necessarily unique,
but the isomorphism types of $\fg$ and $\fk$ are uniquely determined by the
symmetric space $G/K$ (\cite[Theorem 6.1]{helgason}). The roots of $\fg$ whose
root spaces lie in $\fp$ are the \emph{noncompact roots} of $(\fg,\fk)$
(\cite[Section VIII.7]{helgason}). 

Using Cartan's classification of symmetric spaces and symmetric pairs
(\cite[Table X.V]{helgason}), it can be checked case by case that for every
symmetric space $G/K$ whose associated symmetric pair $(\fg,\fk)$ consists of
two simply-laced Lie algebras of the same rank, we may choose $\fk$ in such a
way that there is a QP-triple $(W,I,U)=T(\Gamma,p)$ from Table
\ref{table:tnp_numbers} with the property that $\fg$ is of type $\Gamma$,
$U=U(\Gamma,p)$ consists precisely of the noncompact roots of $(\fg,\fk)$, and
we have $I=S\cap W'$ where $W'$ is the Weyl group of $\fk$. Moreover, every
QP-triple $(W,I,U)=T(\Gamma,p)$ in Table \ref{table:tnp_numbers} satisfies the
above property for a suitable symmetric pair $(\fg,\fk)$ of equal-rank
simply-laced Lie algebras associated to a symmetric space.
\end{rmk}

Table \ref{table:symmspace} details the correspondences between the Cartan types
of the symmetric spaces, the equal-rank simply-laced symmetric pairs, and all
QP-triples from Table \ref{table:tnp_numbers}; see Example
\ref{exa:symmspace}. The second to last column of the
table gives the rank of the symmetric space, which turns out to be equal to the
size $\kappa$ of the elements in $\Omega_U$ in all cases. The last column
describes the decomposition of $\fk$, where we consider the one-dimensional
abelian Lie algebra $\R$ to be of simply-laced Cartan type ``$\R$". Note that
while we have to distinguish the first two rows in Table \ref{table:tnp_numbers}
because the matroid for the QP-triple $T(A_{2k-1},k)$ differs from those
associated with the other QP-triples of type $A$, all the QP-triples of type $A$
correspond to a single type of symmetric space, regardless of the parity of the
rank of $\Phi$ or the value of $p$. Similarly, the QP-triples of type $D$
correspond only to two types of symmetric spaces, even though they correspond to
four distinct types of matroids in Table \ref{table:tnp_numbers}. 

\begin{table}[!ht]
\centering
\begin{center}
\begin{tabular}{ |l|c|c|c|c| }
  \hline
  $G/K$ &$\Gamma$ &  $p$& $\rank(G/K)$  &    $\fk$ \\
  \hline
  AIII & $A_{n}$ & $1, 2, \ldots, n$ & $\min(p,n+1-p)$ & $A_{p-1}+A_{n-p}+\R$\\
  DI & $D_{n}$ & $1, 2, \ldots, \lfloor n/2\rfloor$ & $2p$ & $D_{p}+D_{n-p}$\\
  DIII & $D_{n}$ & $n-1, n$ & $\lfloor n/2\rfloor$ & $A_{n-1}+\R$\\ 
  EII & $E_{6}$ & $2$ &  $4$ & $A_5+A_1$\\
  EIII & $E_{6}$ & $1, 6$ &  $2$ & $D_5+\R$\\ 
  EV & $E_{7}$ & $2,5$ &  $7$ & $A_{7}$\\ 
  EVI & $E_{7}$ & $1$ &  $4$ & $D_6+A_1$\\ 
  EVII & $E_{7}$ & $7$ &  $3$ & $E_6+\R$\\ 
  EVIII & $E_{8}$ & $1,2,5,6$ & $8$ & $D_8$\\ 
  EIX & $E_{8}$ & $8$ & $4$ & $E_7+A_1$\\ 
 \hline
\end{tabular}
\caption{Correspondence between symmetric spaces and QP-triples $T(\Gamma, p)$}
\label{table:symmspace}
\end{center}
\end{table}

\begin{exa}
\label{exa:symmspace}
The QP-triple $(W,I,U):=T(E_8,8)=(W(E_8), S(E_8)\setminus\{s_8\}, U(E_8,8))$
corresponds to the symmetric space of Cartan type EIX. In this case, $\fg$ is
the Lie algebra of type $E_8$, the positive roots of $\fp$ are the roots in $U$,
i.e., the positive roots in which the simple root $\al_8$ appears with odd
coefficient. The positive roots of $\fk$ are the positive roots in which the
simple root $\al_8$ appears with even coefficient, which include the highest
root $\theta$ of $\Phi$. The Weyl group $W'$ of $\fk$ is of type $E_7+A_1$;
the set $I=\{s_1,s_2,\ldots, s_7\}=S\cap W'$ generates the $E_7$ component of
$W'$, while the reflection $s_\theta$, which commutes with all elements of $I$,
generates the $A_1$ component. The common size $\kappa$ of the elements of $\Omega_U$
is $4$, which agrees with the rank of the symmetric space in question. \end{exa}

\begin{rmk}\label{rmk:deform} 
  For any Coxeter system $(W,S)$ and any subset
  $I\se S$, the left cosets of the parabolic subgroup $W_I\le W$ form a
  quasiparabolic $W$-set whose level function $\ell_I$ measures the lengths of
  shortest coset representatives, and using $\ell_I$ one can deform the natural
  action of $W$ on the set $W/W_I$ to construct a module for the Iwahori--Hecke
  algebra $\mathcal{H}$ of $W$. The defining properties of a quasiparabolic set
  $X$ generalize the properties of $W/W_I$ and allow a similar deformation of
  the $W$-action on $X$ \cite[Theorem 7.1]{rains13}. It would be interesting to
  use $k$-roots to study the $\mathcal{H}$-modules obtained via the deformation
  of the quasiparabolic sets $X=\Omega_U$ constructed in this paper. 
\end{rmk}

\subsection*{Acknowledgements}
The second-named author was partially supported by an AMS--Simons Research
Enhancement Grant for PUI Faculty.

\bibliographystyle{plain} \bibliography{gx6.bib}

\end{document}